\documentclass[a4paper,12pt]{amsart}

\usepackage{amsmath,amsthm,amssymb,mathrsfs}
\usepackage{booktabs,graphicx,xcolor,enumerate}
\usepackage{verbatim}
\colorlet{RED}{red}
\usepackage{float}
\usepackage[pagewise]{lineno}
\usepackage[hidelinks]{hyperref}
\usepackage{tikz}
\usetikzlibrary{arrows.meta,positioning}

\newtheorem{theorem}{Theorem}
\newtheorem{lemma}[theorem]{Lemma}
\newtheorem{proposition}[theorem]{Proposition}
\newtheorem{corollary}[theorem]{Corollary}

\newtheorem{remark}[theorem]{Remark}
\numberwithin{equation}{section}
\usepackage[parfill]{parskip}
\usepackage[all]{xy}

\newcommand{\Hom}{\operatorname{Hom}}

\newcommand{\T}{\mathcal T}
\newcommand{\Tone}[1]{%
	\begin{tikzpicture}[scale=1.15]
		\draw[->] (-.25,0) -- (1.55,0);
		\draw[->] (0,-.25) -- (0,1.55);
		\draw[dashed, gray!55] (0,0) -- (1,1);
		\foreach \x/\y in {0/0,0/1,1/1} {
			\fill[gray!35] (\x,\y) circle (2.4pt);
		}
		#1
\end{tikzpicture}}
\begin{document}

\title[Cokernel-closed subcategories]
{On the number of cokernel-closed\\ additive subcategories\\
for uniformly oriented $A_n$ quivers}
\author{Volodymyr Mazorchuk and Shuoyang Xu}

%
%
%

\begin{abstract}
We study the integer sequence that enumerates co\-ker\-nel-closed, 
idempotent split, full, additive subcategories of the category 
of finite  dimensional complex representations of a uniformly 
oriented  $A_n$ quiver. Using a combinatorial model, we describe 
a non-obvious connection of this sequence to Catalan numbers and 
derive an implicit recurrence. We also describe some basic properties
of the lattice underlying the sequences, in particular, we
give an explicit description of the meet irreducible elements of
that lattice.
\end{abstract}

\maketitle

\section{Introduction}

A recent work of Krause and Stoye, see \cite{KS}, studies, 
in particular, different types of additive subcategories of the category
of finite  dimensional complex representations of a uniformly 
oriented  $A_n$ quiver. The conditions that define these
subcategories require that they are closed under some 
categorical constructions, for example, extensions, kernels, 
cokernels and others. This leads naturally to several 
enumeration problems, some of which admit explicit formulae, 
while for others only a few first terms are known. 
Two such mysterious sequences that were identified in
\cite{KS} enumerate extensions closed subcategories and
cokernel closed subcategories. A significant progress in
understanding the number of extension closed subcategories
was made in \cite{Maz}, with some help from Microsoft Copilot.
This includes an explicit summation-type recurrence for the sequence
that allowed for very efficient computations
as well as some asymptotic estimates for its growth.
Julian K{\"u}lshammer formalized the results of
\cite{Maz} in Lean~4 using Aristotle and then used
Aristotle and Claude to significantly improve them, \cite{Kul}.
In particular, it was shown that the sequence in question
satisfies a three-term recurrence with polynomial coefficients
which also allowed to obtain a complete and explicit
formula for its growth (everything automatically formalized
in Lean~4). Note that the recent preprint \cite{En} 
is closely related to another sequence
(with easy enumeration) identified in \cite{KS}.

In the present note we make some first steps in understanding
the other sequence, namely, the one which enumerates
cokernel closed subcategories. At the moment, this one 
looks significantly more complicated.
In more detail, we consider full, additive, idempotent split 
subcategories of the category of finite-dimensional representations 
of a uniformly oriented quiver of type $A_n$. The indecomposable 
representations are indexed by intervals $[a,b]$, with 
$1\leq a\leq b\leq n$. Hence a subcategory we are interested in 
is uniquely determined by a subset of such intervals.
This subset tells us which indecomposable modules our subcategory 
contains.  Our aim is to enumerate such subcategories that are closed 
under cokernels, or, equivalently, under finite colimits.

We look at the problem combinatorially. We arrange the intervals 
as points in a triangular array, so that our subcategory corresponds to 
a subset of this triangle. The condition of being closed under cokernels 
then becomes an explicit closure condition on this subset. To pass from 
a triangle of size $n-1$ to one of size $n$, it is enough to determine 
which points may be added in the new top row. The resulting 
forced containments among the new points are encoded by an oriented graph.

Different subsets of the smaller triangle may admit exactly the 
same choices for possible extensions in the new top row. We therefore 
group such subsets into implication states. We show that each 
implication state can be described  by a threshold vector
$$
t(Y)=(t_0(Y),\ldots,t_n(Y)),
\qquad
t_j(Y)=\min\{i\mid i\preceq_Y j\}.
$$
The possible threshold vectors satisfy a simple intrinsic condition, 
which leads to one of the classical realizations of Catalan numbers. In particular, 
the implication states at level $m$ are counted by the 
Catalan number $C_{m+2}$. The same family of vectors appears in
the combinatorial interpretation of Catalan numbers that can be found
in \cite{Sta}, see also \cite{StaEC}.

The state description also makes it possible to enumerate the original
cokernel-closed subcategories recursively. For each threshold vector $t$, we
keep track of its multiplicity $q_n(t)$, namely the number of closed subsets
giving rise to the corresponding state. We derive a complete recurrence for
these multiplicities. Summing the contributions of all states then yields a
recurrence for the total number of cokernel-closed additive subcategories. 
The first values that we were able to compute are:
\begin{gather*}
2,\ 7,\ 37,\ 265,\ 2396,\ 26118,\ 332723,\ 4845044,\\
79327314,\ 1441840473,\ 28797904948,\ 626815406988,\\
14765197821831,\ 374200770218086,\ 10151746125999166.
\end{gather*}
In particular, this provides nine new values compared to
what was computed in \cite{KS}.

We also investigate various other properties of our sequence.
Out ultimate goal, which we failed to achieve, is to find a closed
and explicit recurrence for our sequence. Various attempts to establish
such a recurrence lead to interesting subsequences. For example,
we can split our sequence into a sum of subsequences with respect to
the number of posits missing in the first column. For small values
for the number of such missing points, we derive an explicit enumeration,
which turns out to be polynomial in $n$. We propose a conjecture that
this polynomial behavior holds for any fixed number of missing points.

We also look into the lattice structure of the set of our subcategories, 
which is partially ordered with  respect to inclusions. We describe its 
atoms and coatoms as well as join irreducible and meet irreducible 
elements. The description of meet irreducible elements is complete, 
aesthetically quite appealing and fairly non-trivial.

The paper is organized as follows. 
Section~\ref{sec:categorical-origin} recalls the representation-theoretic
background and the criterion for cokernel closure. 
In Section~\ref{sec:coordinates}, we introduce
the triangular combinatorial model and the implication graphs associated with
extensions by a new top row.
Section~\ref{sec:catalan} studies implication states and threshold
vectors and establishes their connection with Catalan numbers. 
In Section~\ref{sec:firstc} we investigate the subsequences
defined in terms of the number of points missing in the first column
and enumerate some of them. Here we also provide an enumeration for
the number of elements for which all columns are not empty.
In Section~\ref{sec.recurrence}, we  derive 
recurrences for the state multiplicities and for the total
number of cokernel-closed additive subcategories.
Finally, in Section~\ref{s19}, we present our results 
on atoms, coatoms, join irreducible and meet irreducible elements. 

\setcounter{section}{1}

\section{Representation-theoretic background}
\label{sec:categorical-origin}
\subsection{Quiver and representations}
Let $n$ be a positive integer. Consider the following quiver $A_n$:
$$
A_n:\quad
1\xrightarrow{\alpha_1}2\xrightarrow{\alpha_2}\cdots
\xrightarrow{\alpha_{n-1}}n.
$$
Following \cite{KS}, let
$\operatorname{rep}(A_n)$ be the category of contravariant finite-dimensional representations of $A_n$. An object $V\in \operatorname{rep}(A_n)$ assigns a vector
space $V(i)$ to each vertex and a linear map
$V(\alpha_i):V(i+1)\longrightarrow V(i)$ to each arrow. For $1\le a\le b\le n$, define $V_{a,b}\in \operatorname{rep}(A_n)$ as follows:
$$
\begin{aligned}
	V_{a,b}(i):=\begin{cases}
		\mathbb{C},  & a\le i\le b;\\
		0, & \text{else};
	\end{cases}
\hspace{3em}
	V_{a,b}(\alpha_i):=\begin{cases}
		Id_\mathbb{C},  & a\le i\le b-1;\\
		0, & \text{else}.
	\end{cases}
\end{aligned}
$$
Every indecomposable object of $\operatorname{rep}(A_n)$ is isomorphic to a unique $V_{a,b}$,
and
$$
\Hom_{\operatorname{rep}(A_n)}(V_{a,b},V_{c,d})=
\begin{cases}
	\mathbb C,&a\leq c\leq b\leq d,\\
	0,&\text{otherwise}.
\end{cases}
$$
In particular, $V_{a,b}$ is supported on the closed
interval $[a,b]$. 

\subsection{Cokernel-closed additive subcategories}
Let $\mathcal C\subseteq \operatorname{rep}(A_n)$ be a full, additive,
idempotent split subcategory. We say that $\mathcal C$ is
\emph{cokernel-closed} if $\operatorname{Coker}(f)\in\mathcal C$, for every
morphism $f:X\longrightarrow Y$ with $X,Y\in\mathcal C$.
The following lemma is taken from \cite{KS}, see \cite[Lemma~7]{KS}.

\begin{lemma}
\label{lem:cokernel-formula}
A subcategory $\mathcal C$ as above
is cokernel-closed if and only if it is closed
under the cokernels which have the following forms:
$$
V_{a,b}\xrightarrow{f}V_{c,d}
\qquad\text{and}\qquad
V_{a,b}\xrightarrow{(f,g)}V_{c,d}\oplus V_{c',d'},
$$
where $f$ and $g$ are non-zero. Moreover,
$$
\begin{aligned}
	\operatorname{Coker}(f)&=V_{b+1,d},\\
	\operatorname{Coker}((f,g))&={V_{b+1,\min(d,d')}\oplus V_{\max(c,c'),\max(d,d')}},
\end{aligned}
$$
where, by convention, $V_{p,q}=0$ for $p > q$.
\end{lemma}

\section{Combinatorial counterpart}
\label{sec:coordinates}

Let $n$ be a non-negative integer.  Define 
$$
T(n)=\{(i,j)\in\mathbb Z^2\mid 0\leq i\leq j\leq n\}.
$$
For $n=0,1,2,3$, here is a visual depiction of $T(n)$:

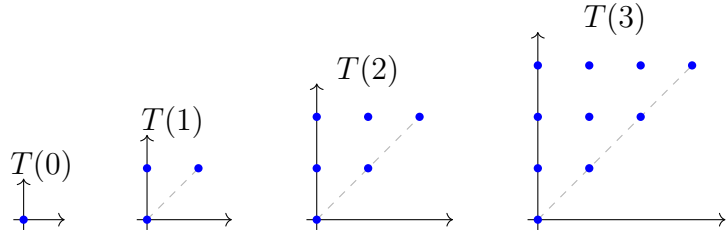
\begin{center}
\begin{tikzpicture}[
  scale=.68,
  point/.style={circle,fill=blue,inner sep=1.1pt},
  axis/.style={->,black},
  boundary/.style={dashed,gray!60}
]
  \begin{scope}[xshift=0cm]
    \draw[axis] (-.2,0)--(.8,0);
    \draw[axis] (0,-.2)--(0,.8);
    \node[point] at (0,0) {};
    \node at (.35,1.05) {$T(0)$};
  \end{scope}

  \begin{scope}[xshift=2.4cm]
    \draw[axis] (-.2,0)--(1.65,0);
    \draw[axis] (0,-.2)--(0,1.65);
    \draw[boundary] (0,0)--(1,1);
    \foreach \i/\j in {0/0,0/1,1/1}{
      \node[point] at (\i,\j) {};
    }
    \node at (.5,1.9) {$T(1)$};
  \end{scope}

  \begin{scope}[xshift=5.7cm]
    \draw[axis] (-.2,0)--(2.65,0);
    \draw[axis] (0,-.2)--(0,2.65);
    \draw[boundary] (0,0)--(2,2);
    \foreach \i/\j in {0/0,0/1,1/1,0/2,1/2,2/2}{
      \node[point] at (\i,\j) {};
    }
    \node at (1,2.9) {$T(2)$};
  \end{scope}

  \begin{scope}[xshift=10cm]
    \draw[axis] (-.2,0)--(3.65,0);
    \draw[axis] (0,-.2)--(0,3.65);
    \draw[boundary] (0,0)--(3,3);
    \foreach \i/\j in {
      0/0,
      0/1,1/1,
      0/2,1/2,2/2,
      0/3,1/3,2/3,3/3
    }{
      \node[point] at (\i,\j) {};
    }
    \node at (1.5,3.9) {$T(3)$};
  \end{scope}
\end{tikzpicture}
\end{center}
For a subset $Y\subset T(n)$, we identify $Y$ with 
the vector $(Y_n,Y_{n-1},\ldots,Y_0)$, where $Y_j$ is the set 
of the first coordinates for those elements in $Y$
which belong to the row indexed by $j$:
$$
Y_j=\{i\mid(i,j)\in Y\}\subseteq\{0,\ldots,j\}.
$$
For example, we identify $Y=\{(0,1),(0,2),(2,2)\}\subset \T(2)$ 
with the vector $(\{0,2\},\{0\},\varnothing)$.

Let $R(n)$ be the set of all subsets $Y$ of $T(n)$ satisfying the following condition:
\begin{align}
	\tag{\(\star\)}
	\label{eq:admissible-star}
	\begin{gathered}
		\text{ for }\quad u=(i,j),\quad v=(k,l),\quad w=(k',l')\in Y,\\
		i\leq k\leq j,\, i\leq k'\leq j,\, j\leq l,\,j\leq l'
		\,\Longrightarrow\,
		(\max(k,k'),\max(l,l'))\in Y,\\
		\text{and, if }j<\min(l,l'),\text{ then also }
		(j+1,\min(l,l'))\in Y.
	\end{gathered}
\end{align}
A graphical illustration of Condition \((\star)\) is given in Figure~\ref{fig:closure-rule}. The presence of the blue points forces the red points to be included.
\begin{figure}[ht]
	\centering
	\begin{tikzpicture}[
		x=1.05cm,y=.85cm,
		chosen/.style={circle,fill=blue!70,inner sep=1.1pt},
		forced/.style={circle,fill=red!75,inner sep=1.1pt},
		gridpoint/.style={circle,fill=gray!35,inner sep=1.1pt},
		lab/.style={font=\small}
		]
		\draw[->,gray!70] (-.35,0)--(5.45,0) node[right] {$i$};
		\draw[->,gray!70] (0,-.35)--(0,5.45) node[above] {$j$};
		\draw[gray!55] (0,0)--(5,5);
		\foreach \j in {0,...,5}{
			\foreach \i in {0,...,\j}{\node[gridpoint] at (\i,\j) {};}
		}
		\node[chosen] at (0,3) {};
		\node[chosen] at (1,5) {};
		\node[chosen] at (2,4) {};
		\node[forced] at (2,5) {};
		\node[forced] at (4,4) {};
		\node[lab,left=3pt] at (0,3) {\resizebox{7mm}{!}{$u=(i,j)$}};
		\node[lab] at (1,5.2) {\resizebox{7mm}{!}{$v=(k,l)$}};
		\node[lab] at (2,4.2) {\resizebox{7mm}{!}{$w=(k',l')$}};
		\node[lab,text=red!70!black] at (2,4.8)
		{\resizebox{15mm}{!}{$(\max(k,k'),\max(l,l'))$}};
		\node[lab,right=3pt,text=red!70!black] at (3.9,4)
		{\resizebox{9mm}{!}{$(j+1,\min(l,l'))$}};
	\end{tikzpicture}
	\caption{}
	\label{fig:closure-rule}
\end{figure}

Set \(R(-1)=\{*\}\), where \(*\) represents the empty initial triangle, and let $r(n)$ be the cardinality of $R(n)$. Note that $r(-1)=1$.

\subsection{Examples}
We have $R(0)=\{\varnothing,\{(0,0)\}\}$, hence $r(0)=2$.
Further, we have

\resizebox{\textwidth}{!}{$
R(1)=\{\varnothing,\{(0,0)\},\{(0,1)\},\{(1,1)\},
\{(0,0),(1,1)\},\{(0,1),(1,1)\},T(1)\}.
$
}

The only missing subset is \(\{(0,0),(0,1)\}\).
Indeed, the presence of $(0,0)$ and $(0,1)$ forces $(1,1)$ 
according to Condition~$(\star)$, as shown in the figure below:
$$
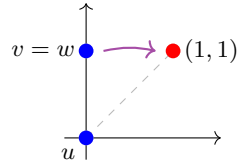
\Tone{
	\coordinate (u) at (0,0);
	\coordinate (v) at (0,1);
	\coordinate (f) at (1,1);
	\fill[blue] (u) circle (2.4pt);
	\fill[blue] (v) circle (2.4pt);
	\fill[red] (f) circle (2.4pt);
	\node[below left] at (u) {\scriptsize \(u\)};
	\node[left] at (v) {\scriptsize \(v=w\)};
	\node[right] at (f) {\scriptsize \((1,1)\)};
	\draw[->, violet!70, thick] (0.2,1) to [bend left=12] (0.8,1);
}
$$
Consequently, $r(1)=7$.

\subsection{Connection}
For $1\le a\le b\le n+1$, there is a bijection between the set of indecomposable objects in $\operatorname{rep}(A_{n+1})$ and $T(n)$, given by:
$$
        V_{a,b}\longmapsto(a-1,b-1).
$$
In this way, $x=V_{i+1,j+1}$, $y=V_{k+1,l+1}$, and $z=V_{k'+1,l'+1}$ correspond to
$u=(i,j)$, $v=(k,l)$, and $w=(k',l')$, respectively.  The cokernel of $x\longrightarrow y\oplus z$ in Lemma~\ref{lem:cokernel-formula} corresponds to
$$
        (j+1,\min(l,l'))
        \quad\text{and}\quad
        (\max(k,k'),\max(l,l')).
$$
The first point occurs precisely when $j<\min(l,l')$.  Thus
Lemma~\ref{lem:cokernel-formula} translates exactly into Condition~$(\star)$.
\subsection{Adding one top layer}
\label{sec:adding-one-top-layer}

Suppose \(Y\in R(n-1)\). To extend \(Y\) to an element of
\(R(n)\), one only has to choose which points in the new top layer are added.
Write this choice as
$$
A\subseteq\{0,\ldots,n\},
\qquad
Y_A=Y\cup\{(e,n)\mid e\in A\}.
$$
Let $S(Y)$ denote the set of all $A\subseteq\{0,\ldots,n\}$
such that $Y_A\in R(n)$ and define an equivalence relation on $R(n-1)$ by
$$
Y\sim_{n-1}Y'\quad\Longleftrightarrow\quad S(Y)=S(Y').
$$
Let \(\sim_{-1}\) be the unique equivalence relation on \(R(-1)\).
Write \([Y]_{n-1}\) for the $\sim_{n-1}$-equivalence class of $Y$.

Let $Z\in R(n)$ and let $Y$ be its restriction to $R(n-1)$.
Then $Z=Y_A$, where $A=Y_n$. In particular, all elements of 
$R(n)$ can be constructed from elements of $R(n-1)$
as described above. Now, given $Y$, let us describe 
the restrictions on $A$ that are forced by Condition~$(\star)$.

\begin{lemma}
	\label{lem:top-layer-extension}
	Let \(Y\in R(n-1)\)
	and $A\subseteq\{0,\ldots,n\}$. Then \(Y_A\in R(n)\) if and
	only if the following two conditions are satisfied.
\begin{enumerate}
\item\label{cond1-1} 
If \((i,j)\in Y\) and \(A\cap\{i,\ldots,j\}\neq\varnothing\), then
\(j+1\in A\).
\item\label{cond1-2} 
If \((i,j),(k,l)\in Y\), with \(i\leq k\leq j\leq l\), and
\(e\in A\cap\{i,\ldots,k-1\}\) then \(k\in A\).
\end{enumerate}
\end{lemma}

\begin{proof}
We only need to check Condition~$(\star)$ in the situation 
when one of the initial points is in the top row. 
In this situation, we immediately see  that
Condition~\eqref{cond1-1} corresponds to the bottom row
of Condition~$(\star)$ while Condition~\eqref{cond1-2} 
corresponds to the middle row of Condition~$(\star)$.
\end{proof}

\subsection{Example}
\label{sec:example-t1-to-t2}

Let $Y=\{(0,0),(0,1),(1,1)\}\in R(1).$
To extend this $Y$ to an element of \(R(2)\), we need to choose
some $A\subseteq\{0,1,2\}$  to get 
$Y_A=Y\cup\{(e,2)\mid e\in A\}$.

If \(0\in A\), the point \((0,0)\in Y\) and the new point 
\((0,2)\) {\color{red}force} \((1,2)\), so \(1\in A\). 
The pair \((0,1),(1,1)\in Y\), together with the new
point \((0,2)\), also {\color{blue}force} \((1,2)\). 
The point \((0,1)\in Y\) and the new point \((0,2)\) 
{\color{teal}force} \((2,2)\), so \(2\in A\). Similarly, if 
\(1\in A\), then \((1,1)\in Y\) {\color{cyan}forces} \(2\in A\). 
This forcing is illustrated in the following  picture:

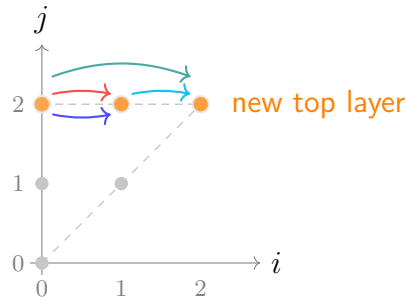
\begin{center}
\begin{tikzpicture}[
	scale=1.05,
	old/.style={gray!45},
	choice/.style={orange!70},
	forced/.style={orange!75},
	lab/.style={font=\scriptsize}
	]
	\draw[->, gray] (-.15,0) -- (2.75,0);
	\draw[->, gray] (0,-.15) -- (0,2.75);
	\node[right] at (2.75,0) {\(i\)};
	\node[above] at (0,2.75) {\(j\)};
	\draw[dashed, gray!55] (0,0) -- (2,2);
	\draw[densely dashed, gray!60] (0,2) -- (2,2);
	\node[right, orange] at (2.25,2) {new top layer};
	
	\foreach \x in {0,1,2} {
		\node[below, gray] at (\x,-.08) {\scriptsize \(\x\)};
	}
	\foreach \y in {0,1,2} {
		\node[left, gray] at (-.08,\y) {\scriptsize \(\y\)};
	}
	
	\foreach \x/\y in {0/0,0/1,1/1} {
		\fill[old] (\x,\y) circle (2.4pt);
	}
	\foreach \x in {0,1,2} {
		\fill[gray!18] (\x,2) circle (3.3pt);
	}
	
	\fill[choice] (0,2) circle (2.6pt);
	\fill[forced] (1,2) circle (2.6pt);
	\fill[forced] (2,2) circle (2.6pt);
\draw[->, red!70, thick, shorten >=4pt, shorten <=4pt] (0,2.1) to [bend left=12] (1,2.1);
\draw[->, blue!70, thick, shorten >=4pt, shorten <=4pt] (0,1.9) to [bend right=12] (1,1.9);
\draw[->, cyan!70, thick, shorten >=4pt, shorten <=4pt] (1,2.1) to [bend left=12] (2,2.1);
\draw[->, teal!70, thick, shorten >=4pt, shorten <=4pt] (0,2.3) to [bend left=20] (2,2.3);
\end{tikzpicture}
\end{center}

This implies that admissible choices of \(A\) over our \(Y\) are 
$$
\varnothing,\qquad
\{2\},\qquad
\{1,2\},\qquad
\{0,1,2\}.
$$
These give four possible extensions of  \(Y\) to an element of \(T(2)\). 
Analogous considerations of all seven elements of \(R(1)\) give 
the total of \(37\) elements in \(R(2)\).

\subsection{Consequences}

The above procedure allows one implicit recursive computations. 
Starting from the empty initial triangle, the
initial family is \(\{\varnothing\}\). Given
\(Y\subseteq T(n-1)\), one can list all 
the subsets \(A\subseteq\{0,\ldots,n\}\) satisfying
the two conditions above. Each such pair gives rise to some
\(Y_A\subseteq T(n)\).

This is equivalent to the direct (brute force) mask enumeration, but 
is more efficient. The direct enumeration for \(T(n)\) checks 
\(2^{(n+1)(n+2)/2}\) subsets. Our recurrence only extends already 
closed subsets and tests \(2^{n+1}\) possible top layers for each
of them.

With this recurrence, one can compute
a few initial values for the sequence $r(n)$, see Figure~\ref{fig1}.
From these values it looks reasonable to guess that 
the sequence $r(n)$ grows as $C\cdot \alpha^n\cdot n!\cdot n^\beta$,
for some constant $C$, where we have $\alpha\in (1.7,1.8)$ and 
$\beta\in (1.3,1.4)$. A bit far fetched but 
interesting speculative observation would be 
that $\sqrt{\pi}\in (1.7,1.8)$.

\begin{figure}
\begin{displaymath}
\begin{array}{c||c}
m&r(m)\\ \hline \hline
0&2\\
1&7\\
2&37\\
3&265\\
4&2396\\
5&26118\\
6&332723\\
7&4845044\\
8&79327314\\
9&1441840473\\
10&28797904948\\
11&626815406988\\
12&14765197821831\\
13&374200770218086\\
14&10151746125999166\\
\end{array}
\end{displaymath}
\caption{Initial values of $r(n)$ up to $n=14$}
\label{fig1}
\end{figure}

\subsection{Implication graph}

For $Y\in R(n-1)$, define a directed graph
\(G_Y=(V_Y,E_Y)\), which we call the
{\em implication graph of $Y$}, as follows:
\begin{itemize}
	\item The vertex set is \(V_Y=\{0,1,\ldots,n\}\).
	\item The arrow set is
\begin{equation}\label{graph:arrows}
	\begin{aligned}
		E_Y={}&\{(e,j+1)\mid (i,j)\in Y,\ i\leq e\leq j\}\\
		&\mathbin{\cup}
		\{(e,k)\mid (i,j),(k,l)\in Y,\
		\ i\leq k\leq j\leq l,\ i\leq e<k\}.
	\end{aligned}
\end{equation}
\end{itemize}
Note that the target of every arrow in $G_Y$ is strictly greater than its source.
We emphasize the following property of $G_Y$ which follows directly
from the definition:
\begin{equation}\label{eq:graphprop}
\text{If } (a,b)\in E_Y, 
\text{ then } (a',b)\in E_Y
\text{ for all } a\leq a'<b.
\end{equation}
By Lemma~\ref{lem:top-layer-extension}, 
we have \(A\in S(Y)\) if and only if $A$ 
is closed under $G_Y$ in the following sense: if \(e\in A\) and \(e\to f\)
in $G_y$, then \(f\in A\). This means that \(S(Y)\) is the 
set of all subsets of $\{0,1,\dots,n\}$ closed under \(G_Y\) in the above sense.

For two directed graphs \(G,H\) on the same set of vertices, we write
$G\sim H$ if and only if the following condition is satisfied:
for every pair $x$ and $y$ of vertices, there is a directed 
path from $x$ to $y$ in $G$ if and only if there is a directed
path from $x$ to $y$ in $H$. Then $\sim$ is an equivalence relation. 
Denote by \([G_Y]\) the equivalence class of $G_Y$ under \(\sim\).
Also, we set $\mathcal S_n=\{[G_Y]\mid Y\in R(n-1)\}$. Write 
\(i\preceq_Y j\) if \(i=j\) or there is a directed path from \(i\) to
\(j\) in \(G_Y\).

From Lemma~\ref{lem:top-layer-extension}, it follows that,
for $Y,Y'\in R(n-1)$, we have $Y\sim_{n-1} Y'$
if and only if $G_Y\sim G_{Y'}$.
Viewing \(R(n-1)\) as the disjoint union of $\sim_{n-1}$-equivalence 
classes, we obtain
\begin{align}
	\label{eq:state-count-by-implication-graphs}
	\vert R(n)\vert=
	\sum_{[G]\in\mathcal S_n}
	N_{n-1}([G])\cdot
	\vert\{A\subseteq\{0,\ldots,n\}\mid A\text{ is closed under }G\}\vert,
\end{align}
where
$  N_{n-1}([G])=\vert\{Y\in R(n-1)\mid G_Y\sim G\}\vert$.

\subsection{Example: \(n=1\)}

We have the following implication graphs for the two
elements in $R(1)$:

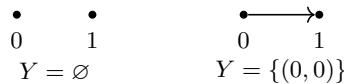
\begin{center}
	\begin{tikzpicture}[
		vertex/.style={circle, fill=black, minimum size=3pt, inner sep=0pt},
		every label/.style={font=\scriptsize, label distance=1pt},
		edge/.style={->, black, semithick},
		lab/.style={font=\scriptsize, align=center}
		]
		\begin{scope}[xshift=0cm,name prefix=n1-g1-]
			\node[vertex, label=below:\(0\)] (v0) at (0,0) {};
			\node[vertex, label=below:\(1\)] (v1) at (1,0) {};
			\node[lab] at (.5,-.75) {\(Y=\varnothing\)};
		\end{scope}
		
		\begin{scope}[xshift=3cm,name prefix=n1-g2-]
			\node[vertex, label=below:\(0\)] (v0) at (0,0) {};
			\node[vertex, label=below:\(1\)] (v1) at (1,0) {};
			\draw[edge] (v0) -- (v1);
			\node[lab] at (.5,-.75) {\(Y=\{(0,0)\}\)};
		\end{scope}
	\end{tikzpicture}
\end{center}

\subsection{Example: \(n=2\)}

We have the following implication graphs for the seven elements in $R(2)$:

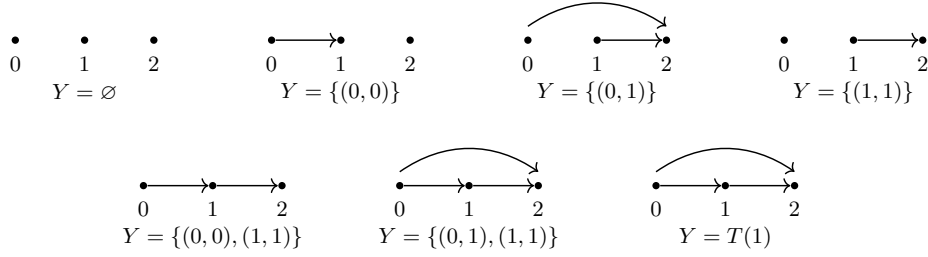
\begin{center}
\resizebox{\textwidth}{!}{
\begin{tikzpicture}[
		vertex/.style={circle, fill=black, minimum size=3pt, inner sep=0pt},
		every label/.style={font=\scriptsize, label distance=1pt},
		edge/.style={->, black, semithick},
		transitive/.style={->, dashed, black, semithick},
		lab/.style={font=\scriptsize, align=center}
		]
		\begin{scope}[xshift=0cm,name prefix=g1-]
			\node[vertex, label=below:\(0\)] (v0) at (0,0) {};
			\node[vertex, label=below:\(1\)] (v1) at (1,0) {};
			\node[vertex, label=below:\(2\)] (v2) at (2,0) {};
			\node[lab] at (1,-.75) {\(Y=\varnothing\)};
		\end{scope}
		
		\begin{scope}[xshift=3.7cm,name prefix=g2-]
			\node[vertex, label=below:\(0\)] (v0) at (0,0) {};
			\node[vertex, label=below:\(1\)] (v1) at (1,0) {};
			\node[vertex, label=below:\(2\)] (v2) at (2,0) {};
			\draw[edge] (v0) -- (v1);
			\node[lab] at (1,-.75) {\(Y=\{(0,0)\}\)};
		\end{scope}
		
		\begin{scope}[xshift=7.4cm,name prefix=g3-]
			\node[vertex, label=below:\(0\)] (v0) at (0,0) {};
			\node[vertex, label=below:\(1\)] (v1) at (1,0) {};
			\node[vertex, label=below:\(2\)] (v2) at (2,0) {};
			\draw[edge] (0,0.2) to[bend left=35] (2,0.2);
			\draw[edge] (v1) -- (v2);
			\node[lab] at (1,-.75) {\(Y=\{(0,1)\}\)};
		\end{scope}
		
		\begin{scope}[xshift=11.1cm,name prefix=g4-]
			\node[vertex, label=below:\(0\)] (v0) at (0,0) {};
			\node[vertex, label=below:\(1\)] (v1) at (1,0) {};
			\node[vertex, label=below:\(2\)] (v2) at (2,0) {};
			\draw[edge] (v1) -- (v2);
			\node[lab] at (1,-.75) {\(Y=\{(1,1)\}\)};
		\end{scope}
		
		\begin{scope}[xshift=1.85cm,yshift=-2.1cm,name prefix=g5-]
			\node[vertex, label=below:\(0\)] (v0) at (0,0) {};
			\node[vertex, label=below:\(1\)] (v1) at (1,0) {};
			\node[vertex, label=below:\(2\)] (v2) at (2,0) {};
			\draw[edge] (v0) -- (v1);
			\draw[edge] (v1) -- (v2);
			\node[lab] at (1,-.75) {\(Y=\{(0,0),(1,1)\}\)};
		\end{scope}
		
		\begin{scope}[xshift=5.55cm,yshift=-2.1cm,name prefix=g6-]
			\node[vertex, label=below:\(0\)] (v0) at (0,0) {};
			\node[vertex, label=below:\(1\)] (v1) at (1,0) {};
			\node[vertex, label=below:\(2\)] (v2) at (2,0) {};
			\draw[edge] (v0) -- (v1);
			\draw[edge] (0,0.2) to[bend left=35] (2,0.2);
			\draw[edge] (v1) -- (v2);
			\node[lab] at (1,-.75) {\(Y=\{(0,1),(1,1)\}\)};
		\end{scope}
		
		\begin{scope}[xshift=9.25cm,yshift=-2.1cm,name prefix=g7-]
			\node[vertex, label=below:\(0\)] (v0) at (0,0) {};
			\node[vertex, label=below:\(1\)] (v1) at (1,0) {};
			\node[vertex, label=below:\(2\)] (v2) at (2,0) {};
			\draw[edge] (v0) -- (v1);
			\draw[edge] (0,0.2) to [bend left=35] (2,0.2);
			\draw[edge] (v1) -- (v2);
			\node[lab] at (1,-.75) {\(Y=T(1)\)};
		\end{scope}
	\end{tikzpicture}
}
\end{center}

Only six out of these seven graphs are different
(the last two graphs coincide).
Furthermore, we only have five $\sim$-equivalence classes
as the last three graphs are equivalent.

\section{Connection to Catalan numbers}\label{sec:catalan}

Recall that the Catalan numbers are defined  by
$C_n=\frac{1}{n+1}\binom{2n}{n}$, for $n\geq 0$.
%
The goal of this section is to prove that the implication states at level
\(m\) are counted by \(C_{m+2}\). To this end, we
introduce certain threshold vectors and show that 
they are connected to one of the classical realizations of
Catalan numbers.

\subsection{Threshold vectors}

For \(Y\in R(n-1)\), define the \(n+1\)-dimensional 
{\em threshold vector}
$$
t(Y)=(t_0(Y),\ldots,t_{n}(Y)),
$$
with $t_j(Y)=\min\{i\mid i\preceq_Y j\}$. Write $t_j=t_j(Y)$. 
We begin with the following lemmata.

\begin{lemma}
	\label{lem:reachability-interval}
	For \(0\leq i\leq j\leq n\), we have
	$$
		i\preceq_Y j
		\quad\Longleftrightarrow\quad
		t_j(Y)\leq i\leq j.
	$$
\end{lemma}

\begin{proof}	
Assume the path from $t_j(Y)$ to $j$ is 
$$
t_j(Y)=x_0\to x_1\to\cdots\to x_r=j.
$$
If \(t_j(Y)\leq i<j\), choose \(h\) such that
\(x_{h-1}\leq i<x_h\). Since
$x_{h-1}\to x_h$ is an arrow, by Property~\eqref{eq:graphprop},
we have the arrow \(i\to x_h\). Combining the arrow \(i\to x_h\) with
the remaining part of the path gives \(i\preceq_Y j\). 
The converse follows from the minimality of \(t_j(Y)\).
\end{proof}

\begin{lemma}
	\label{lem:threshold-bracket-condition}
	The threshold vector \(t(Y)=(t_0(Y),\ldots,t_{n}(Y))\) 
	has the following two properties:
	\begin{align}
		\label{eq:threshold-bracket-condition}
		0\leq t_j(Y)\leq j; \text{ and }\,\,
		(t_j(Y)\leq i\leq j
		\,\Longrightarrow\,
		t_j(Y)\leq t_i(Y)).
	\end{align}
\end{lemma}

\begin{proof}
The property $0\leq t_j(Y)\leq j$ follows
directly from the definitions.
For the second properly, assume
	\(t_j(Y)\leq i\leq j\). Then \(i\preceq_Y j\) by
	Lemma~\ref{lem:reachability-interval}. Since
	\(t_i(Y)\preceq_Y i\preceq_Y j\), the minimality of \(t_j(Y)\) gives
the inequality	\(t_j(Y)\leq t_i(Y)\).
\end{proof}

\begin{lemma}
	\label{lem:threshold-admissible-row}
	A subset \(A\subseteq\{0,\ldots,n\}\) belongs to \(S(Y)\) if and only if
	$$
		(i\in A\,\,\text{ and }\,\, t_j\leq i\leq j)
		\quad\Longrightarrow\quad j\in A.
	$$
\end{lemma}

\begin{proof}
A subset \(A\subseteq\{0,\ldots,n\}\) belongs to \(S(Y)\)
if and only if $A$ is closed under \(G_Y\). This means that 
$i\in A$ and $i\preceq_Y j$ has to imply $j\in A$.
By Lemma~\ref{lem:reachability-interval}, 
the condition \(i\preceq_Y j\) is equivalent to
the condition \(t_j\leq i\leq j\). This proves our lemma.
\end{proof}

For \(A\neq\varnothing\), denote by $a$ the minimum element of
\(A\) and further set 
$\displaystyle\mu_t(a)=\min\{t_j\,\vert\, a\leq j\leq n\}$.
Note that here we take all $j$ between $a$ and $n$,
not only those which are in $A$.

\begin{lemma}
	\label{lem:threshold-transition}
Let \(A\in S(Y)\).
\begin{enumerate}[$($a$)$]
\item\label{claima-2} If \(A\neq \varnothing\), then 
$t_j(Y_A)=
		\begin{cases}
			\mu_t(a),&\text{ if }j\in A\text{ or }j=n+1,\\
			t_j,&\text{ if }j\notin A\text{ and }j\leq n.
		\end{cases}		$
\item\label{claimb-2} If \(A=\varnothing\), then
	$t(Y_A)=(t_0,\ldots,t_{n},n+1)$.
\end{enumerate}
\end{lemma}

\begin{proof}
Assume $A=\varnothing$. In this case 
$G_{Y_A}$ and $G_Y$ have exactly the same arrows
by \eqref{graph:arrows}. In particular,
the vertex $n+1$ is isolated.
This  implies Claim~\eqref{claimb-2}.

Now assume $A\neq \varnothing$.
From \eqref{graph:arrows}, we see that 
$G_{Y_A}$ inherits all arrows from $G_Y$.
In particular, we have $t_j(Y_A)\leq t_j$,
for all $j\leq n$. However, $G_{Y_A}$
also has additional edges, for instance, those 
which correspond to the points $(x,n)\in Y_A$,
where $x\in A$. Each such $(x,n)$ gives the
arrows $(y,n+1)$ in $G_{Y_A}$, for all
$x\leq y\leq n$. These are exactly the arrows
in $G_{Y_A}$ that terminate at $n+1$.
In particular, the minimal $y$ for which there is
an arrow $(y,n+1)$ in $G_{Y_A}$ is, in fact, $y=a$.

Now, choose $h\geq a$  such that $t_h=\mu_t(a)\leq a$.
By definition, there is a path from $\mu_t(a)$ to $h$ in  $G_Y$. 
Since $h\geq a$, by the previous paragraph, there is an arrow from
$h$ to $n+1$ in $G_{Y_A}$ and hence $t_{n+1}(Y_A)\leq \mu_t(a)$.

Next we claim that there is a path from $\mu_t(a)$ to $a$
in $G_{Y_A}$. This is obvious if $\mu_t(a)=a$.
If $\mu_t(a)<a$, consider a path from $\mu_t(a)$ to $h>a$
in $G_Y$. If this path does not go through $a$ at some point, it
has an arrow from $x$ to $y$, for $x<a<y$. 
By \eqref{graph:arrows}, there are two possibilities.

{\bf Case~1.} There is $(i,y-1)\in Y$ with $i\leq x$.
In this case, we consider the points $(i,y-1),(a,n)\in Y_A$. We have 
$i\leq x<a$, then $a\leq y-1$ and $(y-1\leq n)$.
This gives an arrow from $x$ to $a$ in $G_{Y_A}$
by \eqref{graph:arrows}.

{\bf Case~2.} There are $(i,j),(y-1,l)\in Y$ 
such that $i\leq y-1\leq j\leq l$ and $i\leq x$.
Now take $(i,j),(a,n)\in Y$. Then
$i\leq x<a$, moreover $a\leq y-1\leq j$ 
and $j\leq n$. This gives an arrow from
$x$ to $a$ in $G_{Y_A}$ by \eqref{graph:arrows}.

The subpath from $\mu_t(a)$ to $x$ combined with
the arrow from $x$ to $a$ in $G_{Y_A}$ give
a path from $\mu_t(a)$ to $x$ in $G_{Y_A}$.

For any $j\in A$ such that $a<j$, we have $(a,n),(j,n)\in Y_A$
which forces the arrow $(a,j)$ in $G_{Y_a}$ by 
Lemma~\ref{lem:top-layer-extension}.
Therefore we have a path from $\mu_t(a)$
to $j$ in $G_{Y_A}$ and thus $t_j(Y_A)\leq\mu_t(a)$.
Now we need to turn the inequalities we obtained into equalities.

Let us look at the arrows in $G_{Y_A}$ that are not
in $G_A$. These are either of the form $(x,n+1)$,
where $a\leq x$, as mentioned above, or of the form 
$(e,k)$ as in \eqref{graph:arrows}. In the latter case,
they come from some pair $(i,j)(x,n)\in Y_A$
such that $x\in A$ and  $i\leq x\leq j$. In particular,
such an arrow necessarily terminates at $x\in A$.

If $j\not\in A$ and $j\leq n$, then $x\not \preceq_Y j$,
for any $x\in A$, since $A$ is $G_Y$-closed.
Together with the previous paragraph, this means that 
$i\preceq_Y j$ if and only if  $i\preceq_{Y_A}j$,
for all $i\in \{0,1,\dots,n\}$.  Hence $t_j(Y_A)= t_j$
for such $j$.

We already know that, if $i<\mu_t(a)\leq a$,
there are no arrows from $i$ to $n+1$ in $G_{Y_A}$.
Therefore any path from such $i$ to $n+1$
in $G_{Y_A}$ must go via $A$. So, if we prove 
$t_j(Y_A)=\mu_t(a)$, for all $j\in A$, we will
have $t_{n+1}(Y_A)=\mu_t(a)$ automatically.
Therefore it remains to show that 
$t_j(Y_A)=\mu_t(a)$, for all $j\in A$.

Assume now that some $i<\mu_t(a)$ reaches some
element of $A$ in $G_{Y_A}$. Take such $i$ maximal 
possible. Then there should be an arrow from $i$
to some $x\geq \mu_t(a)$ in $G_{Y_A}$.
If $x\not\in A$, this arrow cannot be new 
and hence is already in $G_Y$. 
By \eqref{eq:graphprop}, since we can reach
$h$ from $\mu_t(a)$ in $G_Y$, we can reach
$h$ from $i$ in $G_Y$. This contradicts 
the definition of $\mu_t(a)$.
Hence $x\in A$ and the arrow $(i,x)$
in $G_{Y_A}$ is new.

Our arrow $(i,x)$ in $G_{Y_A}$ 
must come from some $(u,j),(x,n)\in Y_A$
with $u\leq x\leq j\leq n$ and $u\leq i<x$.
Note that $j=n$ is not possible since that would imply
$u\in A$ contradicting $u\leq i<a$. Therefore
$(u,j)\in Y$.
%

The fact that $(u,j)\in Y$ forces, by \eqref{graph:arrows},
arrows $(v,j+1)$, for all values $u\leq v\leq j$, in $G_Y$.
Since $u\leq i$, we have an arrow
$(i,j+1)$ in $G_Y$. But $x\in A$ and $j+1\geq x$,
in particular, $j+1\geq a$. We just computed that $t_{j+1}\leq i$
which contradicts the definition of $\mu_t(a)$.
This proves that $t_s(Y_A)=\mu_t(a)$, for all $s\in A$,
completing the proof.
\end{proof}

For \(n\geq0\), let \(\mathcal T_n\) be the set of vectors
\(t=(t_0,\ldots,t_{n-1})\) satisfying
\eqref{eq:threshold-bracket-condition}. 
By convention, the set \(\mathcal T_0\) consists of the empty vector.

\begin{lemma}\label{lem:catalan}
We have $t\in \mathcal T_n$ if and only if 
$t_j=\vert\{i<j\mid t_i<t_j\}\vert$, for each $0\le j\le n-1$.
\end{lemma}

\begin{proof}
Suppose first that $t\in\mathcal T_n$ and let $t_j=k$.
For $0\leq i<k$, we have $t_i\leq i<k$. For $k\leq i<j$, applying
\eqref{eq:threshold-bracket-condition} gives $k\leq t_i$. Hence
$\{i<j\mid t_i<t_j\}=\{0,\ldots,k-1\}$.
	
Conversely, suppose that $t_j=\vert\{i<j\mid t_i<t_j\}\vert$, for every $j$.
Then $0\leq t_j\leq j$. Let $t_j=k$. Since $t_i\leq i<k$, for
$0\leq i<k$, the $k$ indices $0,1,\dots,k-1$ belong to
$\{i<j\mid t_i<t_j\}$. They are therefore the only such indices. Thus
$k\leq t_i$ whenever $k\leq i<j$, which gives
\eqref{eq:threshold-bracket-condition}.
\end{proof}

Lemma~\ref{lem:catalan} identifies $\mathcal T_n$ 
with the family $(h^6)$ in the Catalan Addendum~\cite{Sta}, see also \cite{HT}.
This implies  the following.

\begin{corollary}
\label{lem:catalan-bracket-vectors}
We have  $|\mathcal T_n|=C_n$. 
\end{corollary}

To each $[G_Y]\in\mathcal S_{n+1}$, we associate a threshold vector
$t(Y)\in\mathcal T_{n+2}$. By Lemma~\ref{lem:reachability-interval},
this $t(Y)$ determines the equivalence class of $G_Y$ uniquely.
In particular,  the map from $\mathcal S_{n+1}$
to $\mathcal T_{n+2}$ that sends $[G_Y]$ to $t(Y)$
is injective.
Proposition~\ref{lem:catalan-bracket-vectors} now gives
$$
\vert\{S(Y)\mid Y\in R(m)\}\vert
=\vert\mathcal S_{m+1}\vert
\leq\vert\mathcal T_{m+2}\vert
=C_{m+2}.
$$

\subsection{States and Catalan numbers}

We are now ready to prove the following result:

\begin{theorem}
\label{thm:catalan-states}
For every \(n\geq0\),
we have  $|\{S(Y)\mid Y\in R(m)\}|=C_{m+2}$. 
\end{theorem}

\begin{proof}
We only need to prove that the map $[G_Y]\mapsto t(Y)$
from $\mathcal S_{n+1}$ to $\mathcal T_{n+2}$ is surjective.
In other words, we need to show that, for every vector 
\(t\in\mathcal T_{n+2}\), there exists $Y\in R(n)$ such that $t(Y)=t$. 
We argue by induction on  \(N=n+2\).  For \(N=2\), the vectors
\((0,1)\) and \((0,0)\) are realised by \(\varnothing\) 
and \(\{(0,0)\}\) in \(R(0)\), respectively.
	
Let \(N\geq3\), take $t=(t_0,\ldots,t_{N-1})\in\mathcal T_N$
and set $k=t_{N-1}$. Consider also 
\(s=(t_0,\ldots,t_{N-2})\in \mathcal T_{N-1}\). By the
induction hypothesis, there is \(Y\in R(N-3)\) with \(t(Y)=s\).
If \(k=N-1\), take \(A=\varnothing\). Then \(t(Y_A)=t\). 

Suppose that \(k<N-1\), and set
	$$
	A=\{i\in\{k,\ldots,N-2\}\mid t_i=k\}.
	$$
	We shall prove $A\in S(Y)$ and $t(Y_A)=t$. Applying \eqref{eq:threshold-bracket-condition} to $t_{N-1}=k$
	gives $k\leq t_i$ for $k\leq i\leq N-1$, in particular, $t_k=k$.
	If $i\in A$ and $t_j\leq i\leq j$, then
	\eqref{eq:threshold-bracket-condition} gives $t_j\leq t_i=k$, while
	$j\geq k$ gives $k\leq t_j$. Thus $t_j=k$ and $j\in A$, so $A\in S(Y)$ by Lemma \ref{lem:threshold-admissible-row}. Moreover, $\mu_s(k)=k$, 
	and hence $t(Y_A)=t$ follows from Lemma~\ref{lem:threshold-transition}.
\end{proof}

All in all, we have now established that, for every \(n\geq1\), 
there are canonical bijections
$$
R(n-1)/{\sim_{n-1}}
\xrightarrow{\sim}\mathcal S_n
\xrightarrow{\sim}\mathcal T_{n+1},
\qquad
[Y]_{n-1}\longmapsto [G_Y]\longmapsto t(Y).
$$

\section{First column statistics}\label{sec:firstc}

\subsection{Definitions and small ranks}\label{sec:firstc.1}

For $n\geq 0$ and $0\leq k\leq n+1$, let $R(n,k)$
denote the subset of $R(n)$ consisting of all $X$
satisfying the condition $|X\cap \{(0,0),(0,1),\dots,(0,n)\}|=k$
(i.e. $X$ has exactly $k$ elements in the first column).
Set $r(n,k):=|R(n,k)|$. Then 
\begin{displaymath}
R(n)=\coprod_{k=0}^{n+1} R(n,k)\quad\text{ and  }\quad
r(n)=\sum_{k=0}^{n+1} r(n,k).
\end{displaymath}
Here is the table of the brute force computed values 
for $r(n,k)$, where $n\leq 6$:
\begin{displaymath}
\begin{array}{c||c|c|c|c|c|c|c|c||c}
n\backslash k & 0& 1 & 2 & 3 & 4 & 5 & 6 & 7&\Sigma\\
\hline\hline
0 & 1 & 1 & & & & & & & 2 \\ 
\hline
1 & 2 & 4 & 1 & & & & & & 7 \\ 
\hline
2 & 7 & 20 & 9 & 1 & & & & & 37 \\
\hline 
3 & 37 & 132 & 79 & 16 & 1 & & & & 265 \\ 
\hline
4 & 265 & 1105 & 784 & 216 & 25 & 1 & & & 2396 \\ 
\hline
5 & 2396 & 11256 & 8957 & 2992 & 480 & 36 & 1 & & 26118\\ 
\hline
6 & 26118 & 135135 & 116975 & 44668 & 8845 & 932 & 49 & 1 & 332723 
\end{array}
\end{displaymath}

\subsection{The value of $r(n,0)$}\label{sec:firstc.2}

Here we prove:

\begin{proposition}\label{prop:rn0}
We have $r(n,0)=r(n-1)$. 
\end{proposition}

\begin{proof}
Given $X\in R(n)$ with no points in the first column,
we can shift $X$ one step to the left to get an element
of $R(n-1)$ and vice versa. This provides a bijection
between $R(n,0)$ and $R(n-1)$ and the claim follows.
\end{proof}

\subsection{The value of $r(n,n+1)$}\label{sec:firstc.3}

Here we prove:

\begin{proposition}\label{prop:rnn+1}
We have $r(n,n+1)=1$. 
\end{proposition}

\begin{proof}
We want to show that $R(n,n+1)=\{T(n)\}$. Indeed,
let us pick some $X\in R(n,n+1)$. Since $X$ contains the whole
first column, the last line of Condition~$(\star)$
applied to the origin and one of the points in the
first column forces $X$ to contain all elements 
on the main diagonal.

If we now apply Condition~$(\star)$ to the origin,
a point on the main diagonal and a higher point
in the first column, we force all the remaining 
points of $T(n)$ to be in $X$. Therefore
$X=T(n)$.
\end{proof}

\subsection{The value of $r(n,n)$}\label{sec:firstc.4}

Here we prove:

\begin{proposition}\label{prop:rnn}
We have $r(n,n)=(n+1)^2$. 
\end{proposition}

\begin{proof}
We split $R(n,n)$ further, according to which 
point in the first column is missing: 
for $0\leq m\leq n$, define $R(n,n,m)$
as the set of all element in $R(n,n)$
which do not contain $(0,m)$. To prove our
claim, it is enough to show that $|R(n,n,m)|=n+1$,
for each $m$.

Arguments similar to those used in the proof of
Proposition~\ref{prop:rnn+1} show that each element
of $R(n,n,m)$ must contain all points outside the
and the
$m+1$-st column. Let us order these points as follows:
\begin{displaymath}
(0,m),(1,m),\dots,(m,m),(m+1,m+1),(m+1,m+2),\dots,(m+1,n).
\end{displaymath}
Note that we have $n+1$ elements in total and that we know
for sure that the point $(0,m)$ is missing.

For the remaining points, it is easy to see that 
Condition~$(\star)$ implies that, as soon as our 
element contains one of the above points,
it must contain all points to the right of it
in the above list.

Let $X_m$ be the subset of $T(n)$ obtained by
removing from $T(n)$ all the above points.
It is also easy to check that if we fix one
of the above points, except for $(0,m)$,
and add to $X_m$ this point and all points to the
right of it in the above list, we get an element
of $R(n,n,m)$. This implies $|R(n,n,m)|=n+1$
and we are done.
\end{proof}

\subsection{The value of $r(n,n-1)$}\label{sec:firstc.5}

Here we prove:

\begin{proposition}\label{prop:rnn-1}
We have $r(n,n-1)=\frac{n^2(n+1)^2}{2}+\binom{n+2}{3}-n$, for $n\geq 1$. 
\end{proposition}

\begin{proof}
Let us analyze some $X\in R(n,n-1)$.
We start by fixing the two missing points $(0,a)$ and $(0,b)$
in the first column, where  $0\le a<b\le n$.
Consider the following two sequences, ordered as written:
\begin{align}
&(0,a),(1,a),\ldots,(a,a),(a+1,a+1),\ldots,(a+1,n),\notag\\
&(0,b),(1,b),\ldots,(b,b),(b+1,b+1),\ldots,(b+1,n).
\label{eq:two-missing-paths}
\end{align}
The two sequences intersect at the single point $z:=(a+1,b)$.

We claim that
every point outside these two paths belongs to $X$. Indeed, for such a
point $(i,j)$ with $i>0$, both $(0,i-1)$ and $(0,j)$ belong to $X$.
Hence Condition~$(\star)$, with the second point used twice,
forces $(i,j)$. For $i=0$, the claim is clear from the assumptions.

Let us temporarily remove the point $z$ from both sequences. 
Just like in Proposition~\ref{prop:rnn}, for each remaining sequence,
we have that, as soon as $X$ contains some point of the sequence, 
$X$ must contain  every point (except possibly $z$) which 
come after it in the sequence.

This means that the intersection with $X$ for each of the two remaining
sequences is specified by the leftmost point of this intersection, 
or by the fact that this intersection is empty.
Each sequence has $n$ points (note that $z$ is removed), and the first point
of each sequence is, by definition, not in $X$. Hence there are $n$ 
choices for each of the two sequences. It remains to
check the restrictions involving $z$.

\medskip
\noindent\textbf{Case 1: $b>a+1$.}

\smallskip
\noindent\textit{Subcase 1.1: $z\in X$.}
Since $z\in X$, the same argument as above gives that 
all points after $z$ in both sequences in
\eqref{eq:two-missing-paths} belong to $X$. It remains to choose the first
selected point in each list. There are $b$ choices in the first list and
$a+1$ in the second, giving $b(a+1)$ choices in total.

\smallskip
\noindent\textit{Subcase 1.2: $z\notin X$.}
Consider the $n^2$ pairs of choices for
the two sequences with $z$ removed. 
If our choice for the first sequence starts before $z$ and 
for the second at one of $(1,b),\ldots,(a,b)$,
then the elements $(a+1,b-1)$ and $(a,b)$;
together with $(0,b-1)$, force $z$ by Condition~$(\star)$.
So, these choice we should not count. There are $b-1$ possible 
starting points for the first sequence and $a$ for the second. 
All other choices satisfy Condition~$(\star)$, giving $n^2-a(b-1)$ 
valid choices. Adding the two subcases, gives
\begin{align}
\label{eq:two-missing-separated}
n^2-a(b-1)+b(a+1)&=n^2+a+b.
\end{align}

\medskip
\noindent\textbf{Case 2: $b=a+1$.}
In this case our special  element $z=(a+1,a+1)$ is on the diagonal.

\smallskip
\noindent\textit{Subcase 2.1: $z\in X$.}
First suppose that the row with the second coordinate \(a\) 
is empty and the row with the second coordinate \(b\) contains only 
\(z\). It remains to see what we can have in
the column with the first coordinate $n+1$ (this one contains $z$) 
and the column with the first coordinate  $n+2$.
If the column containing $z$ contains some $(b,j)$, for $j>b$, 
then $(b+1,j)$ is forced by Condition~$(\star)$ applied
to $z$ and $(b,j)$. Let us pick $j$ minimal. Then 
the other column may start at any of $(b+1,b+1),\ldots,(b+1,j)$,
giving $j-b$ choices. If no point in the column with $z$ belongs to $X$, 
there are $n-b+1$ choices for the other column, including the empty 
choice. These give
\begin{align}
(n-b+1)+\sum_{j=b+1}^{n}(j-b)&=\binom{n-b+2}{2}\notag
\end{align}
configurations.

If at least one of our two rows contains a point in $X$ different 
from $z$, both  columns we are discussing must be contained in $X$. 
Indeed, for the first row, any selected point implies
$(a,a)\in X$. For every $j>b$, the point $(a,j)$ is already present,
so it forces $(b,j)$ together with $(a,a)$, and then $(b+1,j)$ together
with $z$. For the second row, a selected point other than $z$ implies
$(a,b)\in X$. Applying Condition~$(\star)$ to
$(a,b)$, $z$, and $(a,j)$ gives both $(b,j)$ and $(b+1,j)$.
When $a=0$, neither row has such a point, so this alternative is absent.

Once both columns are full, each row has $a+1$ independent choices:
row $a$ may start at any of $(1,a),\ldots,(a,a)$ or be empty, while
row $b$ may start at any of $(1,b),\ldots,(a,b)$ or contain only $z$. Excluding
the choice with no additional row points gives $(a+1)^2-1$
configurations. This subcase therefore contributes
$\binom{n-b+2}{2}+(a+1)^2-1$.

\smallskip
\noindent\textit{Subcase 2.2: $z\notin X$.}
If both rows contain a selected point, then we have $(a,a),(a,b)\in X$.
Condition $(\star)$ would then give $z\in X$, a contradiction.
One checks that this is the only extra restriction
on the two sequences. Each row has $a$ choices for 
its first selected point, so the total number of configurations 
in this subcase is $n^2-a^2$. Adding the two subcases gives
\begin{equation}
\label{eq:two-missing-adjacent}
\resizebox{11cm}{!}{$
n^2-a^2+\binom{n-b+2}{2}+(a+1)^2-1
=n^2+2a+\binom{n-a+1}{2}.
$}
\end{equation}

\medskip
Now, we sum everything up.
First, we go over all pairs $0\le a<b\le n$. For an adjacent pair, the
basic term $n^2+a+b$ equals $n^2+2a+1$, so
by \eqref{eq:two-missing-separated} and \eqref{eq:two-missing-adjacent},
we obtain
\begin{align}
r(n,n-1)
&=\sum_{0\le a<b\le n}(n^2+a+b)
+\sum_{a=0}^{n-1}\left(\binom{n-a+1}{2}-1\right).\notag
\end{align}
There are $\binom{n+1}{2}$ pairs $a<b$, and each integer from $0$ to $n$
occurs in exactly $n$ of the sums $a+b$. Therefore
\begin{align}
\sum_{0\le a<b\le n}(n^2+a+b)
&=n^2\binom{n+1}{2}+n\sum_{j=0}^{n}j
=\frac{n^2(n+1)^2}{2}.\notag
\end{align}
The remaining sum is
\begin{align}
\sum_{a=0}^{n-1}\left(\binom{n-a+1}{2}-1\right)
&=\sum_{j=2}^{n+1}\binom{j}{2}-n
=\binom{n+2}{3}-n,\notag
\end{align}
which implies the proposition.
\end{proof}

\subsection{The value of $r(n,n-2)$}\label{sec:firstc6}

Here we prove:

\begin{proposition}\label{prop:rnn-2}
We have $r(n,n-2)=\frac{n^6+n^5-6n^3-n^2-7n+12}{6}$, for $n\geq 2$. 
\end{proposition}

\begin{proof}
Let $Y\in R(n)$ be such that it has exactly three missing points
in the first column. Let $(0,a),(0,b),(0,c)$, where 
$0\le a<b<c\le n$, be these missing points.
For each $d\in\{a,b,c\}$, consider the ordered path
\begin{displaymath}
(0,d),(1,d),\ldots,(d,d),(d+1,d+1),\ldots,(d+1,n).
\end{displaymath}
There are exactly three points, namely, $x=(a+1,b)$, $y=(a+1,c)$, and $z=(b+1,c)$,
where two the above paths cross. Every point outside these three paths 
belongs to $Y$: for such a point $(i,j)$
with $i>0$, the points $(0,i-1),(0,j)$ are in $Y$ and force $(i,j)$
by Condition~$(\star)$.

Remove the crossings from the three paths. Each remaining list has $n-1$
points, including its prescribed missing first point. Within each list,
any point, if selected, forces all later points. Indeed, within a row, a point
$(i,d)$ forces $(k,d)$ for $i<k$ using $(0,k-1)$; this first-column point
is present when $(k,d)$ is not a crossing. A selected row point $(i,d)$
forces a later column point $(d+1,j)$ using the already present point
$(i,j)$. Finally, a column point $(d+1,j)$ forces $(d+1,l)$ for $l>j$
in the same list, using $(0,j)$ and $(0,l)$.

Number each remaining list from $0$ to $n-2$. 
Let $p,q,r\in\{1,\ldots,n-1\}$ be the first selected positions
in the three lists, respectively. The value $n-1$ means that the list is empty.
These three positions, together with the selected crossings, determine $Y$.
We distinguish four cases according to whether $a,b$ and $b,c$ are adjacent.
In each case we count configurations for fixed $a,b,c$, and only at the end
sum over their possible positions.

\medskip\noindent\textbf{Case 1: $b>a+1$ and $c>b+1$.}

Neither pair of missing rows is adjacent, so none of the crossing points
is on the diagonal. The following table gives the restrictions on
$p,q,r$ for each choice of crossings. All variables remain in
$\{1,\ldots,n-1\}$, and only the crossings in the 
table's first column are selected.
\begin{equation}
\label{eq:path-cuts}
\resizebox{11cm}{!}{
$
\renewcommand{\arraystretch}{1.15}
\begin{array}{c|l}
\text{selected crossings}&\text{restrictions}\\\hline
\varnothing&\neg(p\le b-1,\ q\le a),\quad
 \neg(p\le c-2,\ r\le a),\quad\neg(q\le c-2,\ r\le b-1)\\
\{x\}&p\le b,\ q\le a+1,\ r\ge b\\
\{y\}&p\le c-1,\ q\ge c-1,\ r\le a+1\\
\{x,y\}&\text{impossible}\\
\{z\}&q\le c-1,\ r\le b,\quad
 \neg(p\le b-1,\ q\le a),\quad\neg(p\le c-2,\ r\le a)\\
\{x,z\}&p\le b,\ q\le a+1,\ a<r\le b\\
\{y,z\}&p,q\le c-1,\ r\le a+1,\quad\neg(p\le b-1,\ q\le a)\\
\{x,y,z\}&p\le b,\ q,r\le a+1
\end{array}
$}
\end{equation}
Here $\neg(P,Q)$ means that $P$ and $Q$ cannot both hold at the same time.
To check the table, one uses Condition~$(\star)$ at 
the crossings and at the unselected points of the three lists. 
A missing $x$ excludes either $p\le b-1$ or $q\le a$; a missing 
$y$ excludes either $p\le c-2$ or $r\le a$; a missing $z$ excludes 
either $q\le c-2$ or $r\le b-1$.
A selected crossing forces the suffixes after it in both paths.
If $x$ and $y$ are selected, they force $z$. Combining these 
conditions gives the eight rows of \eqref{eq:path-cuts}.
There are no further restrictions: the two outputs in 
Condition~$(\star)$ are either outside the paths, 
later selected points in these suffixes, or crossings
covered by these conditions.

The counts for all the cases listed in 
\eqref{eq:path-cuts} are as follows. 
\begin{displaymath}
\resizebox{\textwidth}{!}{
$
\renewcommand{\arraystretch}{1.25}
\begin{array}{c|l}
\text{selected crossings}&\text{number of configurations}\\\hline
\varnothing&\begin{aligned}
 &(n-1)^3-(n-1)\bigl(a(b-1)+a(c-2)+(b-1)(c-2)\bigr)\\
 &\quad+a(b-1)^2+a(c-2)^2
\end{aligned}\\
\{x\}&b(a+1)(n-b)\\
\{y\}&(c-1)(a+1)(n-c+1)\\
\{x,y\}&0\\
\{z\}&(n-1)b(c-1)-ab(b-1)-a(c-2)(c-1)+a^2(b-1)\\
\{x,z\}&b(a+1)(b-a)\\
\{y,z\}&(a+1)\bigl((c-1)^2-a(b-1)\bigr)\\
\{x,y,z\}&b(a+1)^2
\end{array}
$}
\end{displaymath}
We now describe in detail the count for each row in \eqref{eq:path-cuts}.

{\bf Row 1: $\varnothing$.}
Start with $(n-1)^3$ triples. Let $E_x,E_y,E_z$ be the sets of
triples satisfying the three excluded pairs in \eqref{eq:path-cuts}, namely,
$E_x: p\le b-1, q\le a$; $E_y: p\le c-2, r\le a$; and
$E_z: q\le c-2, r\le b-1$.
Counting the choices for $p,q,r$ gives
$|E_x|=(n-1)a(b-1)$, $|E_y|=(n-1)a(c-2)$, and $|E_z|=(n-1)(b-1)(c-2)$.
Using $a<b-1<c-2$ to combine the bounds, their intersections have sizes
$|E_x\cap E_y|=a^2(b-1)$, $|E_x\cap E_z|=a(b-1)^2$,
$|E_y\cap E_z|=a(c-2)^2$ and $|E_x\cap E_y\cap E_z|=a^2(b-1)$.
Indeed, $E_x\cap E_y$ requires $p\le b-1$ and $q,r\le a$.
These bounds already imply $q\le c-2$ and $r\le b-1$, so
$E_x\cap E_y=E_x\cap E_y\cap E_z$.
The terms $+|E_x\cap E_y|$ and $-|E_x\cap E_y\cap E_z|$
cancel in inclusion-exclusion. The number of remaining triples is
\begin{multline*}
(n-1)^3-|E_x|-|E_y|-|E_z|+|E_x\cap E_z|+|E_y\cap E_z|=\\
\quad=(n-1)^3-(n-1)\bigl(a(b-1)+a(c-2)+(b-1)(c-2)\bigr)\\
\qquad\quad+a(b-1)^2+a(c-2)^2,
\end{multline*}
which is the first entry of the table.

{\bf Row 2:  $\{x\}$.}
The restrictions $p\le b$, $q\le a+1$, and $r\ge b$ give
$b$, $a+1$, and $n-b$ choices, respectively. Their product is
$b(a+1)(n-b)$.

{\bf Row 3:  $\{y\}$.}
The restrictions $p\le c-1$, $q\ge c-1$, and $r\le a+1$ give
$c-1$, $n-c+1$, and $a+1$ choices, respectively. Their product is
$(c-1)(a+1)(n-c+1)$.

{\bf Row 4:  $\{x,y\}$.}
Taking $(i,j)=x$ and $(k,l)=(k',l')=y$ in Condition~$(\star)$ forces
$(j+1,\min(l,l'))=(b+1,c)=z$. Since $z$ is not selected in this row,
the count is zero.

{\bf Row 5:  $\{z\}$.}
The bounds $q\le c-1$ and $r\le b$ initially give $(n-1)b(c-1)$
triples. To keep $x$ absent, we have to exclude the conjunction of
$p\le b-1$ and $q\le a$; there are $ab(b-1)$ such triples. To 
keep $y$ absent, we have to exclude the conjunction of
$p\le c-2$ and $r\le a$; there are $a(c-2)(c-1)$ such triples.
The overlap is $p\le b-1$ and $q,r\le a$, of size $a^2(b-1)$.
Inclusion-exclusion gives
\begin{displaymath}
(n-1)b(c-1)-ab(b-1)-a(c-2)(c-1)+a^2(b-1).
\end{displaymath}

{\bf Row 6:  $\{x,z\}$.}
The restrictions $p\le b$, $q\le a+1$, and $a<r\le b$ give
$b$, $a+1$, and $b-a$ choices, respectively. Their product is
$b(a+1)(b-a)$.

{\bf Row 7:  $\{y,z\}$.}
The bounds $p,q\le c-1$ and $r\le a+1$ initially give us
$(c-1)^2(a+1)$ triples. To keep $x$ absent,  we have to exclude
the conjunction of $p\le b-1$ and $q\le a$, there are $a(b-1)(a+1)$ 
such triples. Hence the count is $(a+1)((c-1)^2-a(b-1))$.

{\bf Row 8:  $\{x,y,z\}$.}
The restrictions $p\le b$ and $q,r\le a+1$ give
$b(a+1)^2$ triples.

Adding the eight entries gives
\begin{equation}
\label{eq:path-basic}
n^3-3n^2+1+(2n+1)a+2nb+(2n-1)c.
\end{equation}
This is the full count in Case 1. In the other cases, the choices counted
by the same table still satisfy $(\star)$, but additional choices become
possible at diagonal crossings. We count those additional choices separately.

\medskip\noindent\textbf{Case 2: $b=a+1$ and $c>b+1$.}

In this case, $x$ is on the diagonal. This means that, if $x$ 
is selected, it does not need to force all points above it.
The first column of the two paths may start later, provided 
that the second column starts no later than the first. 
If a row point other than $x$ is selected, however, both 
columns are forced as in the nonadjacent case. In terms of 
list positions, the extra choices have $p>b$ and $b\le q\le p$.
Put $g=c-b-1$ and $h=n-c+1$. Only three crossing choices change:
\begin{equation}
\label{eq:path-left-cuts}
\resizebox{11cm}{!}{
$
\renewcommand{\arraystretch}{1.2}
\begin{array}{c|l}
\text{selected crossings}&\text{extra choices}\\\hline
\{x\}&p>b,\ b\le q\le p;\quad r\ge b\text{ if }q\le c-2\\
\{x,z\}&p>b,\ b\le q\le\min(p,c-1),\ r\le b;
 \quad r=b\text{ if }p\le c-2\\
\{x,y,z\}&b<p\le c-1,\ b\le q\le p,\ r\le b
\end{array}
$}
\end{equation}
The bounds involving $c$ come from the two other crossings, exactly as in
\eqref{eq:path-cuts}. For example, in the first row, if $q\le c-2$ then
$z\notin Y$ requires $r\ge b$; if $q\ge c-1$, neither $y$ nor $z$
forces a restriction on $r$.
Counting the choices in \eqref{eq:path-left-cuts} gives, respectively,
\begin{gather*}
(g+h)\left(\binom{g+1}{2}+gh-1\right)+(n-1)\binom{h+1}{2},\\
\binom{g+1}{2}-1+bh(g+1),\quad\text{ and }\quad
b\left(\binom{g+2}{2}-1\right).
\end{gather*}
Indeed, for the first count, we sum over $q=b,\ldots,c-2$ and then
also over $q=c-1,\ldots,n-1$. In the first range there are
$\binom{g+1}{2}+gh-1$ pairs $(p,q)$ and $g+h$ choices for $r$;
in the second there are $\binom{h+1}{2}$ pairs and $n-1$ choices for $r$.
For the second count, split at $p=c-1$; for the third, sum $p-b+1$
over $p=b+1,\ldots,c-1$ and multiply by $b$.
The total correction simplifies to
\begin{equation}
\label{eq:path-left}
(n-1)\binom{n-a}{2}+(n-a)(c-1)-n-1.
\end{equation}
Consequently, the count in Case~2 is 
\eqref{eq:path-basic} plus \eqref{eq:path-left}.

\medskip\noindent\textbf{Case 3: $b>a+1$ and $c=b+1$.}

In this case, $z$ is diagonal. Only the choice $\{z\}$ 
gains extra configurations. Their positions satisfy 
$q>b$, $b\le r\le q$, with $p$ arbitrary.
Writing $h=n-c+1$, the correction is
\begin{equation}
\label{eq:path-right}
(n-1)\sum_{q=b+1}^{n-1}(q-b+1)
=(n-1)\left(\binom{h+1}{2}-1\right).
\end{equation}
For the other choices containing $z$, a selected point to its left forces
both column suffixes, so their counts do not change.
The count in Case 3 is \eqref{eq:path-basic} plus \eqref{eq:path-right}.

\medskip\noindent\textbf{Case 4: $b=a+1$ and $c=b+1$.}

In this case, both $x$ and $z$ are diagonal. 
We write $h=n-c+1$, so $n-1=b+h-1$.
The extra choices beyond \eqref{eq:path-cuts} are listed
below (a semicolon separates disjoint alternatives).
\begin{equation}
\label{eq:path-consecutive-cuts}
\resizebox{11cm}{!}{$
\renewcommand{\arraystretch}{1.15}
\begin{array}{c|l}
\text{selected crossings}&\text{extra choices}\\\hline
\{x\}&p=q=b,\ r<b;\quad p>b,\ b\le q\le p\quad(r\text{ arbitrary})\\
\{y\}&p>b,\ b\le r\le\min(p,q)\\
\{z\}&q>b,\ b\le r\le q\quad(p\text{ arbitrary})\\
\{x,z\}&p=q=b,\ r<b;\quad p>b,\ q=b,\ r\le b;
 \quad p>b,\ b<q\le p,\ b\le r\le q\\
\{y,z\}&p>b,\ b\le r\le q\le p\\
\{x,y,z\}&p>b,\ b\le r\le q\le p
\end{array}
$}
\end{equation}
The choices $\varnothing$ and $\{x,y\}$ do not change.
These restrictions follow by applying the column-start condition at both
diagonal crossings. At $x$, a late start in the first column requires
$b\le q\le p$; at $z$, a late start in the second requires
$b\le r\le q$. If only $y$ is selected, its two source columns instead
require $b\le r\le\min(p,q)$. The alternatives with $p=q=b$ account
for the possibility of selecting points before $z$ when both earlier
row parts are empty. The remaining absent-crossing restrictions are those
in \eqref{eq:path-cuts}.

The counts in the six rows of \eqref{eq:path-consecutive-cuts} are
\[
\renewcommand{\arraystretch}{1.2}
\begin{array}{c|l}
\text{selected crossings}&\text{number of extra configurations}\\\hline
\{x\}&b-1+(n-1)\left(\binom{h+1}{2}-1\right)\\
\{y\}&\dfrac{h(h+1)(2h+1)}6-h\\
\{z\}&(n-1)\left(\binom{h+1}{2}-1\right)\\
\{x,z\}&\binom{h+2}{3}+(b-1)h-1\\
\{y,z\}&\binom{h+2}{3}-1\\
\{x,y,z\}&\binom{h+2}{3}-1
\end{array}
\]
For example, the $\{y\}$ row is counted by
$\displaystyle h(h-1)+\sum_{j=1}^{h-1}j^2$, separating $r=b$ from $r>b$.
The last two rows count weakly ordered triples from $h$ positions,
excluding the triple $(b,b,b)$.
The count of the $\{x,z\}$ row gives
$(b-1)+b(h-1)+\binom{h+2}{3}-h$.
Adding the six rows, the correction equals the sum of
\eqref{eq:path-left} and \eqref{eq:path-right}, evaluated at $c=a+2$,
plus
\begin{equation}
\label{eq:path-both}
\frac{h(h+1)(2h+1)}6+3\binom{h+2}{3}-2h-2.
\end{equation}
The count in Case 4 is the sum of \eqref{eq:path-basic},
\eqref{eq:path-left}, \eqref{eq:path-right}, and \eqref{eq:path-both}.

It remains to add everything up.
First, we sum the common term \eqref{eq:path-basic} over all
$0\le a<b<c\le n$, then add \eqref{eq:path-left} for Cases 2 and 4,
\eqref{eq:path-right} for Cases 3 and 4, and \eqref{eq:path-both} for
Case 4 alone. This counts each configuration once and gives
\begin{displaymath}
\resizebox{\textwidth}{!}{$
\begin{aligned}
\sum_{a=0}^{n-2}\sum_{c=a+2}^{n}
 \left((n-1)\binom{n-a}{2}+(n-a)(c-1)-n-1\right)
 &=\frac{n^5-5n^3+4n}{8},\\
\sum_{b=1}^{n-1}b(n-1)\left(\binom{n-b+1}{2}-1\right)
 &=\frac{n^5+n^4-15n^3+23n^2-10n}{24},\\
\sum_{h=1}^{n-1}\left(\frac{h(h+1)(2h+1)}6+3\binom{h+2}{3}-2h-2\right)
 &=\frac{5n^4+6n^3-29n^2-30n+48}{24}.
\end{aligned}
$}
\end{displaymath}
Adding these three  totals with the forth total
\begin{displaymath}
\resizebox{\textwidth}{!}{$\displaystyle
\sum_{0\le a<b<c\le n}
\left(n^3-3n^2+1+(2n+1)a+2nb+(2n-1)c\right)
=\frac{n^2(n^2-1)(2n^2-1)}{12}
$}
\end{displaymath}
coming from \eqref{eq:path-basic}, gives the stated formula.
\end{proof}

\subsection{Conjecture}\label{sec:firstc.9}

It seems reasonable to conjecture that $r(n,n+1-i)$
is a polynomial in $n$ of degree $2i$ and leading 
coefficient $\frac{1}{i!}$.


\subsection{Non-empty rows}\label{sec:nonempty}

Here we prove:

\begin{proposition}\label{prop:nonempty}
The number of elements $Y\in R(n)$ for which every 
of the $n+1$ rows is nonempty is $(n+1)!$.
\end{proposition}

\begin{proof}
We will show that $Y\subset T(n)$ belongs to $R(n)$ 
and has the property that each row
of $Y$ is not empty if and only if, for every 
$j\in\{0,\ldots,n\}$ there exists 
a unique $m_j\in\{0,\ldots,j\}$ such that
\begin{displaymath}
(i,j)\in Y
\quad\Longleftrightarrow\quad
m_j\leq i\leq j.
\end{displaymath}
This immediately  implies the counting claim.
The sufficiency of the above condition follows directly from
Condition~$(\star)$, as the forcing part of the letter always
refers to a suffix of some row. Let us check  necessity.

Suppose first that $Y\in R(n)$ and every row is nonempty. Fix some element
$j\in\{0,\ldots,n\}$ and assume that $(k,j)\in Y$ for some $k<j$. Since
row $k$ is nonempty, there exists $i\leq k$ such that $(i,k)\in Y$. Applying
Condition~$(\star)$ to $(i,k)$, $(k,j)$ and  $(k,j)$
forces $(k+1,j)\in Y$. Repeating this argument shows that
the intersection of the $j$-th row of $T(n)$ with $Y$, if non-empty, 
is a suffix of that row. That is exactly what we needed to prove.
\end{proof}


\section{Recurrence over the threshold vectors}
\label{sec.recurrence}

\subsection{Closed rows and state multiplicities}

Let $A\subseteq\{0,\ldots,n\}$
and \(t=(t_0,\ldots,t_{n})\in\mathcal T_{n+1}\),
we say that {\em $A$ is closed with respect to $t$}
provided that \(i\in A\) and \(t_j\leq i\leq j\) imply \(j\in A\).
Set
$$
f(t)=\vert\{A\subseteq\{0,\ldots,n\}\mid
A\text{ is closed under }t\}\vert,
$$
and put \(f(())=1\). 

\begin{proposition}
\label{prop:closed-row-count}
Let \(n\geq 0\), \(t\in\mathcal T_{n+1}\), and \(k=t_{n}\). Then
$$
f(t)=f(t_0,\ldots,t_{k-1})
\bigl(1+f(t_k-k,\ldots,t_{n-1}-k)\bigr).
$$
\end{proposition}

\begin{proof}
Set $u=(t_0,\ldots,t_{k-1})$ and 
$v=(t_k-k,\ldots,t_{n-1}-k)$.
The threshold condition for $t_{n}=k$ reads $k\leq t_j$, for
$k\leq j\leq n$. Hence the restrictions of a closed subset to
$\{0,\ldots,k-1\}$ and $\{k,\ldots,n\}$ are independent, and the first
restriction can be any subset closed under $u$, giving $f(u)$ choices.

If $n\notin A$, then $A\cap\{k,\ldots,n\}=\varnothing$, since any
$i$ in this intersection would force $n\in A$. This gives one choice on
the second block. If $n\in A$, the set
$
\{i-k\mid i\in A\cap\{k,\ldots,n\}\}
$
can be any subset closed under $v$, giving $f(v)$ choices.
Conversely, these choices produce every subset closed under $t$. Therefore
$f(t)=f(u)(1+f(v))$.
\end{proof}

For \(n\geq0\) and \(t=(t_0,\ldots,t_n)\in\mathcal T_{n+1}\), define
$$
Q_n(t)=\{Y\in R(n-1)\mid t(Y)=t\}\qquad\text{ and }
\qquad q_n(t)=|Q_n(t)|,
$$
where \(Q_0((0))=R(-1)=\{*\}\) and \(q_0((0))=1\).

For \(t=(t_0,\ldots,t_{n-1},0)\in\mathcal T_{n+1}\) and
\(0\leq a\leq n-1\) with \(t_a=0\), set
$$
Z_a(t)=\{j\in\{a,\ldots,n-1\}\mid t_j=0\}
$$
and
$$
\mathcal P_a(t)=
\left\{s\in\mathcal T_n\ \middle|\
s_j=t_j\text{ for }j\notin Z_a(t),\quad
\min_{j\in Z_a(t)}s_j=0\right\}.
$$

\begin{proposition}\label{prop:state-weight-factorization}
Let \(t\in\mathcal T_{n+1}\).

\begin{enumerate}[$($a$)$]
\item\label{prop:swf2}
If \(t_n=n\), then
\begin{align}
		\label{eq:isolated-terminal-weight}
		q_n(t_0,\ldots,t_{n-1},n)
		=q_{n-1}(t_0,\ldots,t_{n-1}).
\end{align}
\item\label{prop:swf1}
If \(t_n=k\), for \(0<k<n\), then
\begin{align}
		\label{eq:state-weight-factorization}
		q_n(t_0,\ldots,t_n)
		=q_{k-1}(t_0,\ldots,t_{k-1})
		q_{n-k}(t_k-k,\ldots,t_n-k).
\end{align}
\item\label{prop:swf3}
If \(t_n=0\), then
\begin{align}
		\label{eq:terminal-zero-weight-recursion}
		q_n(t)=
		\sum_{\substack{0\leq a\leq n-1\\t_a=0}}
		\ \sum_{s\in\mathcal P_a(t)}q_{n-1}(s).
\end{align}
\end{enumerate}
\end{proposition}

\begin{proof}
If \(t_n=n\), the top row is empty
by Lemma~\ref{lem:threshold-transition}. Deleting this top row 
provides a bijection between the sets $Q_n(t_0,\ldots,t_{n-1},n)$
and $Q_{n-1}(t_0,\ldots,t_{n-1})$. This gives
$q_n(t_0,\ldots,t_{n-1},n)=q_{n-1}(t_0,\ldots,t_{n-1})$
and proves Claim~\eqref{prop:swf2}.

Next we prove Claim~\eqref{prop:swf1}.
Take \(Y\in Q_n(t)\). We claim that $Y$ does not
contain any $(i,j)$ with $i<k\leq j+1$. Indeed, 
such a point would force a path $i\rightsquigarrow j+1$ in $G_Y$.
Since $t_n=k\leq j+1\leq n$, Lemma~\ref{lem:reachability-interval} 
then gives $j+1\preceq_Y n$, and hence $i\preceq_Y n$, 
contradicting the minimality of $k=t_n$. Hence $Y$ has the form
\begin{displaymath}
\xymatrix@R=0.5mm@C=0.5mm{
&&&&k&&&& \\
\ast&\ast&\ast&\ast&\bullet&\ast&\ast&\ast&\ast \\
{\color{red}\circ}&{\color{red}\circ}&{\color{red}\circ}&{\color{red}\circ}&
{\color{blue}\ast}&{\color{blue}\ast}&{\color{blue}\ast}&{\color{blue}\ast}&\\
{\color{red}\circ}&{\color{red}\circ}&{\color{red}\circ}&{\color{red}\circ}&
{\color{blue}\ast}&{\color{blue}\ast}&{\color{blue}\ast}&&\\
{\color{red}\circ}&{\color{red}\circ}&{\color{red}\circ}&{\color{red}\circ}&
{\color{blue}\ast}&{\color{blue}\ast}&&&\\
{\color{red}\circ}&{\color{red}\circ}&{\color{red}\circ}&{\color{red}\circ}&
{\color{blue}\ast}&&&&\\
{\color{red}\circ}&{\color{red}\circ}&{\color{red}\circ}&{\color{red}\circ}&
&&&&\\
{\color{violet}\ast}&{\color{violet}\ast}&{\color{violet}\ast}&&
&&&&\\
{\color{violet}\ast}&{\color{violet}\ast}&&&
&&&&\\
{\color{violet}\ast}&&&&
&&&&
}
\end{displaymath}
Here we just established that the {\color{red}red} rectangle is empty.
Now, set
$$
	\begin{aligned}
		Y^{\mathrm L}
		&=Y\cap\{(i,j)\mid 0\leq i\leq j\leq k-2\},\\
		Y^{\mathrm R}
		&=\{(i-k,j-k)\mid (i,j)\in Y,\ k\leq i\leq j\leq n-1\}.
	\end{aligned}
$$
There correspond to 
the {\color{violet}violet} and the {\color{blue}blue} triangles
in our picture, respectively. These two parts do not interact under
Condition~$(\star)$, since every horizontal
interval in $Y^{\mathrm L}$ ends before any horizontal
interval in $Y^{\mathrm R}$ begins. The threshold vectors 
for these parts are
$(t_0,\ldots,t_{k-1})$ and $(t_k-k,\ldots,t_n-k)$, respectively.
Then the coordinate decomposition gives a bijection
	$$
	Q_n(t)\longleftrightarrow
	Q_{k-1}(t_0,\ldots,t_{k-1})
	\times Q_{n-k}(t_k-k,\ldots,t_n-k),
	$$
which implies \eqref{eq:state-weight-factorization}.

Finally, we prove Claim~\eqref{prop:swf3}.
	Fix \(Y\in Q_n(t)\). Since \(t_n=0\), the top row
	\(A=Y_{n-1}\) is nonempty. Set \(a=\min A\), then
	\(X=Y\cap T(n-2)\) and \(s=t(X)\). By
	Lemma~\ref{lem:threshold-transition}, we have
	$$
	\min_{a\leq j\leq n-1}s_j=0\qquad\text{ and }
	\qquad
	t_j=
	\begin{cases}
		0,&j\in A,\\
		s_j,&j\notin A.
	\end{cases}
	$$
	In particular, \(A\subseteq Z_a(t)\). If
	\(j\in Z_a(t)\setminus A\), then \(s_j=t_j=0\). Since \(a\in A\),
	Lemma~\ref{lem:threshold-admissible-row} implies \(j\in A\), a
	contradiction. It follows that
	$
	A=Z_a(t).
	$
	Therefore \(s\in\mathcal P_a(t)\).
	
	Conversely, take \(s\in\mathcal P_a(t)\), then \(X\in Q_{n-1}(s)\) and set
	\(A=Z_a(t)\). We have $t_j\leq s_j$ for every $j$, so every subset closed under $t$ is also closed under $s$. The set $A$ is closed under $t$: if $i\in A$ and
	$t_j\leq i\leq j$, then \eqref{eq:threshold-bracket-condition} gives
	$t_j\leq t_i=0$, and hence $j\in A$. Thus $A$ is also closed under $s$,
	so $A\in S(X)$. Since $\displaystyle\min_{a\leq j\leq n-1}s_j=0$,
	Lemma~\ref{lem:threshold-transition} gives $t(X_A)=t$.
	
	Therefore we have a bijection:
	$$
	\{Y\in Q_n(t)\mid \min Y_{n-1}=a\}
	\longleftrightarrow
	\coprod_{s\in\mathcal P_a(t)}Q_{n-1}(s).
	$$
	Summing over all possible values of \(a\) proves
	\eqref{eq:terminal-zero-weight-recursion}.
\end{proof}

For $n\geq 0$, define 
$\displaystyle
b_n=\sum_{\substack{t\in\mathcal T_{n+1}\\t_n=0}}q_n(t)$. 
A few first values for the sequence $b$
are: $1,1,4,24,189$.

\subsection{Recurrence}

\begin{theorem}
	\label{thm:state-count-recurrences}
We have 
	\begin{align}
		\label{eq:weighted-state-total}
		r_n=\sum_{t\in\mathcal T_{n+1}}q_n(t)f(t)
		\qquad(n\geq0),
	\end{align}
	and
	\begin{align}
		\label{eq:primitive-weight-convolution}
		r_{n-1}=\sum_{t\in\mathcal T_{n+1}}q_n(t)=b_n+\sum_{k=1}^{n}r_{k-2}b_{n-k}
		\qquad(n\geq0).
	\end{align}
\end{theorem}

\begin{proof}
Equation~\eqref{eq:weighted-state-total} follows from
\eqref{eq:state-count-by-implication-graphs}. To prove
\eqref{eq:primitive-weight-convolution}, we
split the sum according to \(t_n=0\) and \(t_n=k>0\), 
and apply Proposition~\ref{prop:state-weight-factorization}.
\end{proof}

Theorem~\ref{thm:state-count-recurrences} provides a
recurrence for \(r_n\). This recurrence is implicit as it
is defined in terms of the elements of $\mathcal T_{n+1}$.

\subsection{Computational comparison}

For comparison, direct mask enumeration tests
$2^{(n+1)(n+2)/2}$ subsets of $T(n)$. The running times below were measured
in the same Python environment.

\begin{center}
\resizebox{\textwidth}{!}{
\begin{tabular}{c|r|r|r|r}
$n$&subsets tested&direct enumeration&formula \eqref{eq:weighted-state-total}&speed-up\\ \hline
4&$32{,}768$&$0.384$ s&less than $0.001$ s&more than $380$\\
5&$2{,}097{,}152$&$22.9$ s&$0.001$ s&about $20{,}000$\\
6&$268{,}435{,}456$&at least $50$ min${}^*$&$0.005$ s&more than $5\cdot10^5$\\
7&$6.87\cdot10^{10}$&not feasible&$0.024$ s&--
\end{tabular}
}
\end{center}

The top-row enumeration is substantially better than direct enumeration: at
the last step it tests $r_{n-1}2^{n+1}$ rows. The state recurrence is still
considerably faster.

\begin{center}
\resizebox{\textwidth}{!}{
\begin{tabular}{c|r|r|r|r}
$n$&rows tested&top-row enumeration&formula \eqref{eq:primitive-weight-convolution}&speed-up\\ \hline
5&$153{,}344$&$0.092$ s&$0.004$ s&$23$\\
6&$3{,}343{,}104$&$1.76$ s&$0.021$ s&$84$\\
7&$85{,}177{,}088$&$39.3$ s&$0.11$ s&about $350$\\
8&$2.48\cdot10^9$&at least $18$ min${}^*$&$0.61$ s&more than $1{,}800$
\end{tabular}
}
\end{center}

Here ${}^*$ marks a conservative estimate from the number of candidates; all
other displayed times were measured directly. The speed-up grows because the
state recurrence stores one weight for each threshold state instead of listing
all elements of $R(n)$.

\section{Lattice structure}\label{s19}

\subsection{Basics, atoms, coatoms and join irreducible elements}\label{s19.1}

The set $R(n)$ has the obvious structure of a complete meet semilattice
given by the intersection. Consequently, $R(n)$, being finite, is a lattice.
The join is given by the intersection of all elements of $R(n)$ that contain
all inputs of the join. As all singletons belong to $R(n)$, 
they are exactly the atoms and, as a lattice,
$R(n)$ is both atomic and atomistic. The join irreducible elements
are the atoms.

From Proposition~\ref{prop:rnn+1} and Proposition~\ref{prop:rnn}, 
it follows that the 
coatoms are exactly the elements $T(n)\setminus\{(0,i)\}$,
for $i=0,1,\dots,n$. In particular, we have $n+1$ coatoms.

\subsection{Meet irreducible elements}\label{s19.2}

Similarly to the situation studied in \cite{Maz}, description of 
the meet irreducible elements of $R(n)$ is much more complicated.

\subsubsection{First family}\label{s19.2.1}

For any $(i,j)\in T(n)$, define $N(i,j)$ as $T(n)$ with the points
$(0,j),(1,j),\dots,(i,j)$ removed. Note that we have, in total, 
$\binom{n+2}{2}$ elements of this form.
Here is an example: the element $N(1,2)$ for $n=3$
(the points inside $N(1,2)$ are represented by $\bullet$
and the points outside $N(1,2)$ are represented by $\circ$):
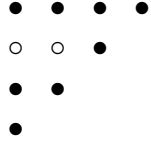
\begin{displaymath}
\xymatrix@C=1mm@R=1mm{
\bullet&\bullet&\bullet&\bullet\\ 
\circ&\circ&\bullet&\\ 
\bullet&\bullet&&\\ 
\bullet&&&\\ 
}
\end{displaymath}

\begin{lemma}\label{lem:mi-1}
Each $N(i,j)$ is a meet irreducible element of $R(n)$.
\end{lemma}

\begin{proof}
Each  $N(i,j)$ belongs to $R(n)$ as each row of $N(i,j)$ is a suffix.
If we add to $N(i,j)$ some $(s,j)$, then Condition~$(\star)$
forces the elements $(s+1,j)$, then $(s+2,j)$ and so on, all the way up to
$(i,j)$. This means that $N(i,j)\cup\{(i,j)\}$ is the only element of
$R(n)$ which covers $N(i,j)$. Therefore $N(i,j)$ is meet irreducible.
\end{proof}

\subsubsection{Second family}\label{s19.2.2}

Choose integers $1\leq d\leq a\leq c\leq b\leq n$.
Define the northeast rectangle
\[
Q_{\mathrm{NE}}=
\{(x,y)\in T(n):d\leq x\leq a,\ c\leq y\leq b\}
\]
and the southwest rectangle
\[
Q_{\mathrm{SW}}
=\{(x,y)\in T(n):0\leq x\leq d-1,\ d-1\leq y\leq c-1\}.
\]
Both are genuine axis-parallel rectangles contained in $T(n)$.
The north-east corner of $Q_{\mathrm{SW}}$
is $\operatorname{NE}(Q_{\mathrm{SW}})=(d-1,c-1)$
and the south-west corner of $Q_{\mathrm{NE}}$
is $\operatorname{SW}(Q_{\mathrm{NE}})=(d,c)$,
so the rectangles are diagonally adjacent.  Set
\[
M(a,b,c,d)=T(n)\setminus
\bigl(Q_{\mathrm{SW}}\cup Q_{\mathrm{NE}}\bigr).
\]
Note that we have, in total, 
we have exactly $\binom{n+3}{4}$ elements of the form $M(a,b,c,d)$.
Here is an example:  $M(3,5,4,2)$ for $n=6$:
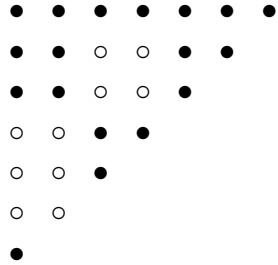
\begin{displaymath}
\xymatrix@C=1mm@R=1mm{
\bullet&\bullet&\bullet&\bullet&\bullet&\bullet&\bullet\\ 
\bullet&\bullet&\circ&\circ&\bullet&\bullet&\\ 
\bullet&\bullet&\circ&\circ&\bullet&\\  
\circ&\circ&\bullet&\bullet\\ 
\circ&\circ&\bullet&\\ 
\circ&\circ&&\\ 
\bullet&&&\\ 
}
\end{displaymath}

\begin{lemma}\label{lem:mi-2}
Each $M(a,b,c,d)$ is a meet irreducible element of $R(n)$.
\end{lemma}

\begin{proof}
Let $M=M(a,b,c,d)$.  We start by proving that 
$M(a,b,c,d)$ belongs to $R(n)$. Note that if we look at 
the bottom part of $M(a,b,c,d)$ between the rows
$0$ through $c-1$, each row in this part is a suffix.
Therefore we cannot apply Condition~$(\star)$
to create new points in $Q_{\mathrm{SW}}$ from
outside of it. Therefore the only possible potential
problems should be related to creation of 
new points in $Q_{\mathrm{NE}}$ from the rest of $M$.
However, if we take three (not necessarily different) 
points satisfying 
the hypotheses in  Condition~$(\star)$ which create
a new point in $Q_{\mathrm{NE}}$, the point that is 
most south-west out of these three must belong to
$Q_{\mathrm{SW}}$, by the definition of the latter.
Since $Q_{\mathrm{SW}}$ is completely removed from $M$,
it follows that $M\in R(n)$.

To prove that $M$ is meet irreducible, it is enough to
show that the element $M\cup\{(a,b)\}$ is the only cover of $M$.
Note that $(a,b)$ is the most north-east point of $T(n)$
that is not in $M$. Hence, adding it to $M$, produces an element
of $R(n)$. In other words, $M\cup\{(a,b)\}$ is indeed a
cover of $M$. To prove uniqueness, we need to show that 
addition of any other point inside 
$Q_{\mathrm{NE}}\cup Q_{\mathrm{SW}}$ to $M$ eventually forces
$(a,b)$ by Condition~$(\star)$.

Inside $Q_{\mathrm{NE}}$, adjoining one point propagates 
successively to the right and upward which eventually terminates 
at $(a,b)$.  Each elementary step is one of the following 
two instances of Condition~$(\star)$, namely: 
$(u,v)\rightsquigarrow(u+1,v)$ and $(u,v)\rightsquigarrow(u,v+1)$
where the other two points used lie immediately outside the
removed rectangle and hence belong to $M$.  
Moreover, any point $(i,j)\in Q_{\mathrm{SW}}$ forces
$(d,c)\in Q_{\mathrm{NE}}$ by applying Condition~$(\star)$
to $(i,j)$, $(i,c)$ and $(d,j)$, where all three points 
are in $M$. This completes the proof.
\end{proof}

\subsubsection{Classification}\label{s19.2.3}

\begin{theorem}\label{thm:mi}
Every meet irreducible element of $R(n)$
is of the form $N(i,j)$, for some
$0\leq i\leq j\leq n$, or of the form
$M(a,b,c,d)$, for some $1\leq d\leq a\leq c\leq b\leq n$.
\end{theorem}

\begin{proof}
Let $Y\in R(n)$ be meet irreducible. Consider
$D:=T(n)\setminus Y\neq\varnothing$. If $(a,b)\in D$
is such no $(u,v)\in D\setminus\{(a,b)\}$ 
satisfies $a\leq u$ and $b\leq v$ at the same time, 
then $Y\cup\{(a,b)\}\in R(n)$
and hence is a cover of $Y$. Therefore the meet
irreducibility of $Y$ implies that such $(a,b)$
is unique, moreover, every $(i,j)\in D$ 
eventually forces $(a,b)$ relative to $Y$
by Condition~$(\star)$. Consequently, 
the whole of $D$ is contained in 
$\{(i,j)\,:\,i\leq a\text{ and }j\leq b\}$.

Let us now look at the $b$-th row of $D$ which is 
a subset of the set $\{(0,b),(1,b),\dots,(a,b)\}$.
Let $d\leq a$ be minimal such that  $(d,b)\in D$.
We seek to show that the $b$-th row of $D$ 
is  $\{(d,b),(d+1,b),\dots,(a,b)\}$.
Assume this is not the case and some 
$(i,b)$ is missing from $D$, where $d<i<a$.
Choose $i$ maximal with such property and
let $j$ be maximal such that $j<i$
and $(j,b)\in D$. Since $Y$ is meet irreducible,
adding $(j,b)$ to $Y$ eventually forces $(a,b)$.
Since $Y$ contains everything in all rows above
row $b$, adding $(j,b)$ to $Y$ can only force
new elements in row $b$, that is, elements to the
right of $(i,b)$. However, if $(j,b)$ together with
some $p$ and $q$ force such a new element, then 
$(i,b)$ together with
the same $p$ and $q$ force the same new 
element, contradicting $Y\in R(n)$. Therefore,
indeed, the $b$-th row of $D$ 
is exactly the set $\{(d,b),(d+1,b),\dots,(a,b)\}$,
for some $d\leq a$.

Now we need to consider the case $d=0$ separately.
In this case, we have $Y\subset N(a,b)$ and
$Y\subset (Y\cup\{(a,b)\})$ with
$Y=N(a,b)\cap (Y\cup\{(a,b)\})$. Applying the
meet irreducibility of $Y$, we deduce
$Y=N(a,b)$, so this case is done.

Therefore, for the remainder of the proof, we assume $d>0$.
Assume first that $a=b$. In this case 
the fact that $Y\in R(n)$ and Condition~$(\star)$
imply that the points $(0,b-1),(1,b-1),\dots,(d-1,b-1)$
all must be in $D$. If $a=d$, we obtain
$Y\subset M(a,a,a,a)$. Assume $Y\neq M(a,a,a,a)$.
Then $D$ contains some point in some row strictly below
row $b-1$. Choose a point $p$ in $D$ that is most
north-east out of all points in $D$ in rows 
between $0$ and $b-2$. Then adding $p$ to $Y$
produces an element $Y'\in R(n)$.
Indeed, $p$ cannot force anything new below
row $b-1$ as it is the most north-east point in that
region. It cannot force any new point in row $b-1$
as this row is empty in $Y$. And $p$ cannot force
the diagonal point $(a,a)$ as $p$ is strictly below row $a-1$.
Consequently, $Y$ is the meet of $M(a,a,a,a)$
and $Y'$ which means that $Y= M(a,a,a,a)$  due to 
meet irreducibility.

Now consider the case $d<a$. 
Note that $c=a=b$ and hence we have
$Q_{\mathrm{NE}}=\{(d,a),(d+1,a),\dots,(a,a)\}$. 
Moreover, we already established that the
top row of $Q_{\mathrm{SW}}$ is in $D$.
In fact, a similar argument shows that
$D$ contains $Q_{\mathrm{SW}}$. As any element of
$Q_{\mathrm{SW}}$ together with one of 
$(0,b),\dots,(d-1,b)$ forces one of
$(d,b),\dots,(a,b)$ by Condition~$(\star)$.

Now let us look at the triangular region
of $T(n)$ that is east of $Q_{\mathrm{SW}}$.
Assume $D$ has a point in that region.
In this case let $p$ be some extreme 
north-east point of $D$ in that region.
Since $Q_{\mathrm{SW}}\subset D$,
adding $p$ to $Y$ cannot trigger
Condition~$(\star)$ and hence produces
a cover of $Y$ in $R(n)$. 
As $Y=(Y\cup\{(a,b)\})\cap(Y\cup\{p\})$,
we get a contradiction with the meet-irreducibility
of $Y$. Therefore $D$ cannot intersect this
triangular region at all.

Similarly we obtain that $D$ cannot intersect
the triangular area below $Q_{\mathrm{SW}}$.
Consequently, $Y=M(a,a,a,d)$ in this case.

It remains to consider the case $a<b$.
Note that the whole row $d-1$ of $T(n)$ belongs to $D$.
Indeed, any element of that row together
with $(d-1,b)\in Y$ forces $(d,b)$ by Condition~$(\star)$.
Define $c$ as the minimal element such that
$d\leq c$ and $(d-1,c)\in Y$. Note that 
$c\leq b$. Applying Condition~$(\star)$
to $(d-1,c)$, $(d-1,b)$ and any elements in 
$Q_{\mathrm{NE}}$ forces one of the elements in 
$(d,b),\dots,(a,b)$. Therefore $Q_{\mathrm{NE}}\subset D$.

Next we claim that $c\geq a$. Indeed, if $c<a$, then 
$(d-1,c)\in Y$ and $(d-1,b)\in Y$ force
$(c+1,b)\in Y$ by Condition~$(\star)$,
where $d\leq c+1\leq a$, which is not possible.
In particular, our $a$, $b$, $c$ and $d$
satisfy the defining conditions for $M(a,b,c,d)$.

Next, any $p\in Q_{\mathrm{SW}}$ together with
$(d-1,c)$ forces a point in $Q_{\mathrm{NE}}$
by Condition~$(\star)$. Hence
$Q_{\mathrm{SW}}\subset D$ and we obtain
$Y\subset M(a,b,c,d)$. Now we want to prove that
$Y= M(a,b,c,d)$.

Consider the following depiction of our $Y$.
We know for sure that the elements in 
$Q_{\mathrm{SW}}$ and $Q_{\mathrm{NE}}$
are absent, they are depicted by $\circ$.
The elements we know for sure are present are depicted
by $\bullet$. We also have three potentially
undecided areas with elements depicted by $*$
shown in {\color{violet}violet},
{\color{blue}blue} and {\color{teal}teal}
\begin{displaymath}
\xymatrix@C=1mm@R=1mm{
\bullet&\bullet&\bullet&\bullet&\bullet&\bullet&\bullet&\bullet&\bullet\\ 
\bullet&\bullet&\bullet&\circ&\circ&\bullet&\bullet&\bullet\\ 
{\color{violet}*}&{\color{violet}*}&{\color{violet}*}&\circ&\circ&\bullet&\bullet&\\ 
{\color{violet}*}&{\color{violet}*}&{\color{violet}*}&\circ&\circ&\bullet&\\  
\circ&\circ&\circ&{\color{blue}*}&{\color{blue}*}\\ 
\circ&\circ&\circ&{\color{blue}*}&\\ 
\circ&\circ&\circ&&\\ 
{\color{teal}*}&{\color{teal}*}&&&\\ 
{\color{teal}*}&&&\\ 
}
\end{displaymath}

We start with the {\color{blue}blue} area.
To trigger Condition~$(\star)$ with a blue 
element involved, all involved elements must be blue
and the outcome will also be blue.
Therefore, if $D$ intersects this area, we can pick
some extreme north-east element $p$ in this intersection
and get that $Y\cup\{p\}\in R(n)$ is another cover of $Y$
(in addition to $Y\cup\{(a,b)\}$).
This contradicts our assumption that $Y$ is meet
irreducible. Consequently, the whole blue area
belongs to $Y$.

Now we can play a similar game 
with the {\color{violet}violet} area.
To trigger Condition~$(\star)$ with a violet 
element involved, all involved elements must be 
either violet or teal and the outcome will
be violet. The same argument as in the previous paragraph
gives that the whole violet area
belongs to $Y$.

Finally, the same argument works for the teal area
as well implying that the whole teal area
belongs to $Y$. This proves that 
$Y= M(a,b,c,d)$ and we are done.
\end{proof}

\begin{corollary}\label{thm:mi-count}
The number of meet irreducible elements in 
$R(n)$ is given by  $\binom{n+3}{4}+\binom{n+2}{2}$.
\end{corollary}


\begin{remark}
{\em 
Julian K{\"u}lshammer formalized the results of
the original version of this paper in Lean~4 using Aristotle, 
\cite{Kul}. He also fed the original version of the
paper to ChatGPT 6 Astra and asked to improve. 
This eventually resulted in a 26 pages long pdf output with several
claims about the asymptotic growth of our sequence
based on our threshold-state recurrence. These include, 
for example, the following exact claim (lean verified): 
$\displaystyle \lim_{n\to \infty}\sqrt[n]{\frac{r_{n-1}}{n!}}=\rho^{-1}$,
where $\rho=2\log 2 -\frac{\pi^2}{12}$. Unfortunately, most 
of that output is way too advanced for our understanding.
We are happy to provide the files upon request.
}
\end{remark}

\subsection*{Acknowledgements}
The first author is partially supported by the Swedish Research Council.
The second author is partially supported by the Fundamental Research 
Funds for the Central Universities, grant No.~B250201222, and by the Young 
Faculty Overseas Training Program, jointly funded by Hohai University
and the China Scholarship Council. This research was carried out during 
the visit of the second author to Uppsala University,
whose hospitality is gratefully acknowledged.
We thank Julian K{\"u}lshammer for helpful conversations, comments
as well as for the lean formalization and improvements.

We have used Microsoft Copilot and Chat GPT+ extensively
for mathematical discussions and, especially, for help with 
computations and coding during the exploratory stages of this project
as well as for checking the results and the text at various stages.
We have also used SageMath, \cite{Sage}, for computations.

\vspace{1cm}

\noindent
V.~M.: Department of Mathematics, Uppsala University, Uppsala, Sweden\\
{\tt mazor@math.uu.se}

S.~X.:  School of Mathematics, Hohai University, Nanjing, China\\
{\tt shuoyang@hhu.edu.cn}

\newpage

{\bf SageMath code 1.} Below we provide a SageMath code, written
by Microsoft Copilot, that computes the values in Figure~\ref{fig1}.
\vspace{2mm}

\begin{verbatim}

import os
import sys
import subprocess
import tempfile
from pathlib import Path

CPP_SOURCE = r'''#include <bits/stdc++.h>
using U=unsigned __int128; using Big=unsigned long long;
using namespace std;
struct VH{size_t operator()(vector<U> const&v)const{size_t 
h=1469598103934665603ULL;for(U x:v){h^=x;h*=1099511628211ULL;}return h;}};
struct UH{size_t operator()(U x)const{return (uint64_t)x ^ 
((uint64_t)(x>>64)*0x9e3779b97f4a7c15ULL);}};
struct Imp{U a,r;};
vector<pair<int,int>> pts(int n){vector<pair<int,int>> p;for(int 
j=0;j<=n;j++)for(int i=0;i<=j;i++)p.push_back({i,j});return p;}
vector<Imp> build(int n){auto p=pts(n);map<pair<int,int>,int> ix;for(int 
q=0;q<(int)p.size();q++)ix[p[q]]=q;unordered_map<U,U,UH>D;for(int 
x=0;x<(int)p.size();x++){auto [i,j]=p[x];vector<int> V;for(int 
y=0;y<(int)p.size();y++){auto[k,l]=p[y];
if(i<=k&&k<=j&&j<=l)V.push_back(y);}for(int 
y:V)for(int z:V){auto[k,l]=p[y];auto[a,b]=p[z];U 
ant=(U(1)<<x)|(U(1)<<y)|(U(1)<<z);U 
req=U(1)<<ix[{max(k,a),max(l,b)}];
if(j<min(l,b))req|=U(1)<<ix[{j+1,min(l,b)}];
req&=~ant;if(req)D[ant]|=req;}}vector<Imp> 
out;for(auto [a,r]:D)out.push_back({a,r});return out;}
bool good(U s,vector<Imp>const& im){for(auto [a,r]:im)if((s&a)==a && 
(s&r)!=r)return false;return true;}
vector<U> signature(U s,int n, vector<Imp>const& next){int 
old=(n+1)*(n+2)/2,q=n+2,N=1<<q; vector<U> z((N+63)/64);for(int 
t=0;t<N;t++)if(good(s|(U(t)<<old),next))z[t>>6]|=U(1)<<(t&63);return z;}
bool allowed(vector<U>const& sg,int t){return (sg[t>>6]>>(t&63))&1;}
struct State{U rep;Big cnt;vector<U> sg;};
\end{verbatim}
\newpage

\begin{verbatim}
int main(int argc, char** argv){
  int target = 9;
  if(argc > 1) target = stoi(argv[1]);
  if(target < 0 || target > 12){
    cerr << "Please choose a target between 0 and 12.\n";
    return 2;
  }
  vector<vector<Imp>> ims(target+2);
  for(int n=0;n<=target+1;n++) ims[n]=build(n);
  vector<State> cur;
  for(U s=0;s<2;s++) cur.push_back({s,1,signature(s,0,ims[1])});
  cout << "n=0 r=2 states=" << cur.size() << endl;
  for(int n=0;n<target;n++){
    int old=(n+1)*(n+2)/2, q=n+2;
    unordered_map<vector<U>,int,VH> pos;
    vector<State> nxt;
    for(auto const& st:cur){
      int N=1<<q;
      for(int t=0;t<N;t++) if(allowed(st.sg,t)){
        U ns=st.rep|(U(t)<<old);
        auto nsg=signature(ns,n+1,ims[n+2]);
        auto it=pos.find(nsg);
        int k;
        if(it==pos.end()){
          k=nxt.size();
          pos.emplace(nsg,k);
          nxt.push_back({ns,0,move(nsg)});
        } else k=it->second;
        nxt[k].cnt += st.cnt;
      }
    }
    Big total=0;
    for(auto const& x:nxt) total += x.cnt;
    cout << "n=" << n+1 << " r=" << total << " states=" 
    << nxt.size() << endl;
    cur=move(nxt);
  }
}
'''
\end{verbatim}
\newpage

\begin{verbatim}

def main():
     # When loaded from the Windows Sage console, sys.argv may contain
     # Sage's own flags such as "-i".  Therefore we do not read sys.argv.
     target = 9
     build_dir = Path.home() / ".cache" / "admissible_transfer"
     build_dir.mkdir(parents=True, exist_ok=True)
     cpp_file = build_dir / "admissible_transfer.cpp"
     executable = build_dir / "admissible_transfer"

     old = cpp_file.read_text() if cpp_file.exists() else None
     if old != CPP_SOURCE:
         cpp_file.write_text(CPP_SOURCE)

     needs_compile = (not executable.exists() or
                      executable.stat().st_mtime < cpp_file.stat().st_mtime)
     if needs_compile:
         print("Compiling the C++ transfer engine ...", flush=True)
         cmd = ["g++", "-O3", "-march=native", "-std=c++17",
                str(cpp_file), "-o", str(executable)]
         subprocess.run(cmd, check=True)

     print("Computing through n=%d ..." % target, flush=True)
     subprocess.run([str(executable), str(target)], check=True)

if __name__ == "__main__":
     main() 
\end{verbatim}

\newpage

{\bf SageMath code 2.} Below we provide a SageMath code, written
by Chat GPT+, that uses \eqref{eq:weighted-state-total}.
\vspace{2mm}

\begin{verbatim}
import argparse
from collections import defaultdict
from functools import lru_cache
from math import comb
from time import perf_counter
ROW_CACHE_LENGTH = 8
def _make_closed_masks(t):
length = len(t)
if length == 0:
return (0,)
k = t[-1]
left_masks = closed_masks(t[:k])
right = tuple(x - k for x in t[k:-1])
right_masks = closed_masks(right)
top = 1 << (length - 1)
without_top = left_masks
with_top = tuple(left | (right_mask << k) | top
for left in left_masks
for right_mask in right_masks)
return without_top + with_top
@lru_cache(maxsize=None)
def _small_closed_masks(t):
return _make_closed_masks(t)
def closed_masks(t):

if len(t) <= ROW_CACHE_LENGTH:
return _small_closed_masks(t)
return _make_closed_masks(t)
@lru_cache(maxsize=None)
def closed_row_count(t):
if len(t) == 0:
return 1
k = t[-1]
left = t[:k]
right = tuple(x - k for x in t[k:-1])
return closed_row_count(left) * (1 + closed_row_count(right))
def terminal_zero_child(s, mask, last_zero):
a = (mask & -mask).bit_length() - 1
if a > last_zero:
return None
return tuple(0 if (mask >> j) & 1 
else s[j] for j in range(len(s))) + (0,)
def positive_children(q_levels, n):
previous = q_levels[n - 1]
result = {}
for prefix in previous:
for k in range(1, n + 1):
if k < n and min(prefix[k:]) < k:
continue
t = prefix + (k,)
left = t[:k]
right = tuple(x - k for x in t[k:])
result[t] = q_levels[k - 1][left] * q_levels[n - k][right]
return result
def catalan_number(n):
return comb(2 * n, n) // (n + 1)
def verify_level(n, q, a_values, b_values):
assert len(q) == catalan_number(n + 1)
previous_total = 1 if n == 0 else a_values[n - 1]
assert sum(q.values()) == previous_total
b_n = sum(weight for t, weight in q.items() if t[-1] == 0)
convolution = b_n
for k in range(1, n + 1):
a_k_minus_2 = 1 if k == 1 else a_values[k - 2]
convolution += a_k_minus_2 * b_values[n - k]
assert convolution == previous_total
b_values.append(b_n)
def compute_a(N, show_stats=False):
q_levels = [{(0,): 1}]
a_values = []
b_values = []
start = perf_counter()
if show_stats:
print(" n states admissible rows a_n seconds")
for n in range(N + 1):
q = q_levels[n]
verify_level(n, q, a_values, b_values)
a_n = 0
row_count = 0
zero_weights = defaultdict(int)
for t, weight in q.items():
f_t = closed_row_count(t)
row_count += f_t
a_n += weight * f_t
if n < N:
masks = closed_masks(t)
assert len(masks) == f_t
last_zero = max(j for j, value in enumerate(t) if value == 0)
for mask in masks:
if mask == 0:
continue
child = terminal_zero_child(t, mask, last_zero)
if child is not None:
zero_weights[child] += weight
a_values.append(a_n)
if show_stats:
print("{:2d} {:11d} {:17d} {:19d} {:10.3f}".format(
n, len(q), row_count, a_n, perf_counter() - start))
if n < N:
next_level = dict(zero_weights)
next_level.update(positive_children(q_levels, n + 1))
q_levels.append(next_level)
known = [2, 7, 37, 265, 2396, 26118, 332723, 4845044, 79327314]
assert a_values[: min(len(a_values), len(known))] == known[: len(a_values)]
return a_values
def parse_arguments():
parser = argparse.ArgumentParser(
description="Compute a_0, ..., a_N using the threshold-state recursion.")
parser.add_argument("N", nargs="?", type=int, 
help="largest index to compute")
parser.add_argument(
"--stats", action="store_true", 
help="show state counts and cumulative timings")
args = parser.parse_args()
if args.N is None:
args.N = int(input("N = "))
if args.N < 0:
parser.error("N must be non-negative")
return args
if __name__ == "__main__":
arguments = parse_arguments()
values = compute_a(arguments.N, arguments.stats)
print("[a_0, ..., a_{}] = {}".format(arguments.N, values))
\end{verbatim}

\newpage

{\bf SageMath code 3.} Below we provide a SageMath code, written
by Chat GPT+, that uses \eqref{eq:primitive-weight-convolution}.
\vspace{2mm}

\begin{verbatim}
import argparse
from collections import defaultdict
from functools import lru_cache
from math import comb
from time import perf_counter
ROW_CACHE_LENGTH = 8
def _make_closed_masks(t):
length = len(t)
if length == 0:
return (0,)
k = t[-1]
left_masks = closed_masks(t[:k])
right = tuple(x - k for x in t[k:-1])
right_masks = closed_masks(right)
top = 1 << (length - 1)
return left_masks + tuple(
left | (right_mask << k) | top
for left in left_masks
for right_mask in right_masks)
@lru_cache(maxsize=None)
def _small_closed_masks(t):
return _make_closed_masks(t)
def closed_masks(t):
if len(t) <= ROW_CACHE_LENGTH:
return _small_closed_masks(t)
return _make_closed_masks(t)
def terminal_zero_states(previous):
result = defaultdict(int)
for s, weight in previous.items():
last_zero = max(j for j, value in 
enumerate(s) if value == 0)
for mask in closed_masks(s):
if mask == 0:
continue
a = (mask & -mask).bit_length() - 1
if a > last_zero:
continue
t = tuple(
0 if (mask >> j) & 1 else s[j] 
for j in range(len(s)) ) + (0,)
result[t] += weight
return dict(result)

def positive_states(q_levels, n):
previous = q_levels[n - 1]
result = {}
for prefix in previous:
for k in range(1, n + 1):
if k < n and min(prefix[k:]) < k:
continue
t = prefix + (k,)
left = t[:k]
right = tuple(x - k for x in t[k:])
result[t] = q_levels[k - 1][left] * q_levels[n - k][right]
return result
def next_q_level(q_levels, n):
result = terminal_zero_states(q_levels[n - 1])
result.update(positive_states(q_levels, n))
return result
def catalan_number(n):
return comb(2 * n, n) // (n + 1)
def compute_a(N, show_stats=False):
q_levels = [{(0,): 1}]
b_values = [1]
a_values = []
start = perf_counter()
if show_stats:
print(" m q-level states b_(m+1) a_m seconds")
for n in range(1, N + 2):
q_n = next_q_level(q_levels, n)
q_levels.append(q_n)
assert len(q_n) == catalan_number(n + 1)
b_n = sum(weight for t, weight in q_n.items() if t[-1] == 0)
b_values.append(b_n)
# Formula (13): a_{n-1}=b_n+sum_{k=1}^n a_{k-2}b_{n-k}.
a_n_minus_1 = b_n
for k in range(1, n + 1):
a_k_minus_2 = 1 if k == 1 else a_values[k - 2]
a_n_minus_1 += a_k_minus_2 * b_values[n - k]
assert a_n_minus_1 == sum(q_n.values())
a_values.append(a_n_minus_1)
if show_stats:
print("{:2d} {:9d} {:11d} {:19d} {:19d} {:10.3f}".format(
n - 1,n,len(q_n),b_n,a_n_minus_1,perf_counter() - start,))
known = [2, 7, 37, 265, 2396, 26118, 332723, 4845044, 79327314]
assert a_values[: min(len(a_values), len(known))] == known[: len(a_values)]
return a_values
def parse_arguments():
parser = argparse.ArgumentParser(
description="Compute a_0, ..., a_N from formula (13), without f(t).")
parser.add_argument("N", nargs="?", type=int, 
help="largest index to compute")
parser.add_argument(
"--stats", action="store_true", 
help="show state counts and cumulative timings")
args = parser.parse_args()
if args.N is None:
args.N = int(input("N = "))
if args.N < 0:
parser.error("N must be non-negative")
return args
if __name__ == "__main__":
arguments = parse_arguments()
values = compute_a(arguments.N, arguments.stats)
print("[a_0, ..., a_{}] = {}".format(arguments.N, values))

\end{verbatim}

\end{document}